\documentclass[12pt,a4]{amsart}
\usepackage[a4paper, left=28mm, right=28mm, top=28mm, bottom=34mm]{geometry}
\allowdisplaybreaks

\usepackage{amssymb
,amsthm
,amsmath
,amscd
,mathtools	
,mathdots
,leftidx
,color
}

\usepackage[all]{xy}
\usepackage{amsmath}
\usepackage{amssymb}
\usepackage{amsthm}
\usepackage{mathrsfs}
\usepackage{mathtools}
\usepackage{url}
\usepackage[dvipdfmx]{graphicx}
\usepackage{amsfonts,amscd}
\usepackage[dvipdfmx]{graphicx}
\usepackage{comment}

\newcommand{\GL}{\mathrm{GL}}

\newcommand{\SO}{\mathrm{SO}}

\newcommand{\GSp}{\mathrm{GSp}}

\newcommand{\GSpin}{\mathrm{GSpin}}

\newcommand{\oo}{\mathcal{O}}
\newcommand{\p}{\mathfrak{p}}

\newcommand\Z{\mathbb{Z}}
\newcommand\A{\mathbb{A}}

\newcommand\Q{\mathbb{Q}}
\newcommand\R{\mathbb{R}}
\newcommand\C{\mathbb{C}}

\newcommand\K {\mathbb{K}}

\def\PP{{\mathbb P}}
\newcommand\et{\textup{\'et}}

\DeclareMathOperator{\Aut}{Aut}

\newcommand\Ker{\textup{Ker}}
\newcommand\Gal{\mathrm{Gal}}
\newcommand\Shaf{\mathrm{Shaf}}

\newcommand\Sh{\mathrm{Sh}}
\newcommand\Pic{\mathop{\mathrm{Pic}}\nolimits}

\def\Spec{\mathop{\mathrm{Spec}}\nolimits}
\newcommand\Frac{\mathrm{Frac}}

\newcommand\tr{\mathrm{tr}}

\DeclareMathOperator{\Bir}{Bir}

\DeclareMathOperator{\ord}{ord}

\def\polL{\mathscr{L}}
\def\polM{\mathscr{M}}

\newcommand\id{\mathrm{id}}

\theoremstyle{plain}

\theoremstyle{definition}

\newtheorem*{lemma*}{Lemma}
\newtheorem*{prop*}{Proposition}
\newtheorem*{theorem*}{Theorem}
\newtheorem*{claim*}{Claim}
\newtheorem{definition*}{Definition}

\theoremstyle{plain}
\newtheorem{theoremsub}[subsection]{Theorem}
\newtheorem{propsub}[subsection]{Proposition}
\newtheorem{lemmasub}[subsection]{Lemma}

\theoremstyle{definition}

\newtheorem{definitionsub}[subsection]{Definition}
\newtheorem{remarksub}[subsection]{Remark}

\newtheorem{defn-propsub}[subsection]{Definition-Proposition}

\AtBeginDocument{%
  \def\MR#1{}
}

\begin{document}
\title[Shafarevich conjecture]{On the Shafarevich conjecture for irreducible symplectic varieties}
\author[T.\ Takamatsu]{Teppei Takamatsu}

%\date{\today}
\address{Department of Mathematics, Graduate School of Science, Kyoto University, Kyoto 606-8502, Japan}
\email{teppeitakamatsu.math@gmail.com}

\maketitle

\begin{abstract}
Irreducible symplectic varieties are higher-dimensional analogues of K3 surfaces.
In this paper, we prove the Shafarevich conjecture for irreducible symplectic varieties of fixed deformation class.
We also observe that the second cohomological generalization of the Shafarevich conjecture does not hold in general, and discuss another formulation of a cohomological generalization.
\end{abstract}
\setcounter{section}{0}
\section{Introduction}
\label{sectionintro}
The Shafarevich conjecture for abelian varieties, which is proved by Faltings and Zarhin, states the finiteness of isomorphism classes of abelian varieties of a fixed dimension over a fixed number field admitting good reduction away from a fixed finite set of finite places.

In this paper, we shall prove an analogue of the Shafarevich conjecture for irreducible symplectic varieties.
Irreducible symplectic varieties are higher-dimensional analogues of
$K3$ surfaces. They play a significant role in algebraic geometry since they are known to be one of the building blocks of (weak) Calabi-Yau varieties (cf.\,Beauville--Bogomolov decomposition \cite{Beauville1983}, \cite{Bogomolov1974}), though there are very few known examples of irreducible symplectic varieties (namely, $K3^{[n]}$-type, generalized Kummer type, $OG_{6}$-type, and $OG_{10}$-type).
Recently, the arithmetic properties of irreducible symplectic varieties
have been studied by many people (see \cite{Andre1996}, \cite{Yang2019}, and \cite{Bindt2021}).

We fix a finitely generated normal $\Z$-algebra $R \subset \C$.
Let $F := \Frac R$ be the field of fractions of $R$.
It is a finitely generated extension of $\Q$ equipped with an embedding into $\C$.
For an irreducible symplectic variety $X$ over $F$
and a height $1$ prime ideal $\mathfrak{p} \in \Spec R$,
we say $X$ admits \textit{essentially good reduction} at $\mathfrak{p}$
if there exists a finite extension $F'/F$ which is unramified at $\mathfrak{p}$,
a birational map
$f_{\mathfrak{p}} \colon Y \overset{\simeq}{\dashrightarrow} X \otimes_F F'$
of irreducible symplectic varieties over $F'$,
and a proper smooth algebraic space
$\mathscr{Y} \to \Spec R'_{\mathfrak{p}}$ with
$\mathscr{Y} \otimes_{R'_{\mathfrak{p}}} F' \simeq Y$.
Here $R'_{\mathfrak{p}}$ is the valuation ring of $F'$ at
a prime above $\mathfrak{p}$.
We note that the extension $F'/F$ and a birational map $f_{\mathfrak{p}}$
may depend on $\mathfrak{p}$.
Also, we do not require a polarization on $X$ (or $Y$) extends
to a polarization on the integral model $\mathscr{Y}$.
It turns out that this seemingly artificial condition is
a natural condition for irreducible symplectic varieties.
We note that if $X$ admits essentially good reduction at $\mathfrak{p}$, then the $\ell$-adic cohomology
$H^2_{\et}(X_{\overline{F}}, \Q_{\ell})$
is unramified at $\mathfrak{p}$ when the residue characteristic at $\mathfrak{p}$
is different from $\ell$.
The Shafarevich conjecture for irreducible symplectic varieties is the following.

\begin{theoremsub}[See Theorem \ref{mainthm} for precise statements]
\label{thmintromain}
Let $X_0$ be an irreducible symplectic variety over $\C$ with the second Betti number $b_{2}(X_{0}) \geq 5$.
There exist only finitely many isomorphism classes of
irreducible symplectic varieties $X$ over $F$ such that
\begin{itemize}
\item $X$ admits essentially good reduction for 
any height $1$ prime ideal $\mathfrak{p} \in \Spec R$.
\item $X \otimes_F \C$ is deformation equivalent to $X_0$.
\end{itemize}
\end{theoremsub}

Here, we note that we fix the deformation class of irreducible symplectic varieties, since the boundedness question for irreducible symplectic varieties is still open.
Since every K3 surfaces are deformation equivalent, Theorem \ref{thmintromain} is a direct generalization of the Shafarevich conjecture for K3 surfaces (\cite{Andre1996}, \cite{She2017}, \cite{Takamatsu2020a}).
We also note that we use a little different formulation in Theorem \ref{mainthm}. Theorem \ref{thmintromain} follows from Theorem \ref{mainthm} and Remark \ref{remnumdef}.

In Theorem \ref{thmintromain}, we do not assume the existence of a polarization on
integral models extending a given polarization on $X$, and we do not need the assumption on polarization degree.
In this sense, our result is a generalization of the result of Andr\'e (\cite[Theorem 1.3.1]{Andre1996}) for very polarized irreducible symplectic varieties. 
Note that the assumption $b_{2}(X_{0}) \geq 5$ is technical, but this assumption is satisfied for all the known examples of irreducible symplectic varieties.

On the other hand, it is natural to consider whether a cohomological generalization,
as studied in \cite{Takamatsu2020a}, holds true or not.
Actually, we can show that a direct cohomological generalization does not hold for generalized Kummer varieties, using non-trivial automorphisms acting trivially on the second cohomology.
Combining with Theorem \ref{thmintromain}, we can also show that good reduction criterion of the formulation given by \cite{Liedtke2018} does not hold for generalized Kummer varieties (see Proposition \ref{propgenkum}).

However, it can be shown that a cohomological generalization holds true
for certain classes of irreducible symplectic varieties,
including $K3^{[n]}$-type varieties.
\begin{theoremsub}[Cohomological generalization, see Theorem \ref{cohshaf}]
\label{thmintromain2}
Let $X_{0}$ be an irreducible symplectic variety over $\C$ with the second Betti number $b_{2}(X_{0}) \geq 5$.
Moreover, assume that $\Aut(X_0) \to H^2(X_0,\Q)$ is injective.
Let $\ell$ be a prime number invertible in $R$.
Then there exist only finitely many isomorphism classes of
irreducible symplectic varieties $X$ over $F$ such that
\begin{itemize}
\item the action of $\Gal(\overline{F}/F)$ on
$H^2_{\et}(X_{\overline{F}}, \Q_{\ell})$
is unramified at
any height $1$ prime ideal $\mathfrak{p} \in \Spec R$, and
\item $X \otimes_K \C$ is deformation equivalent to $X_0$.
\end{itemize}
%(Need to assume $b_2(X_0) > 3$? Need to fix $\ell$?)
\end{theoremsub}
\begin{remarksub}
If $X_{0}$ is $K3^{[n]}$-type or $OG_{10}$-type, then assumptions of Theorem \ref{thmintromain2} are satisfied. We can also formulate another generalization for the case of generalized Kummer varieties, by seeing third cohomologies (see Remark \ref{remcohshaf}).
\end{remarksub}
We will sketch the strategy of the proof of Theorem \ref{thmintromain}.
We note that there exists a Beauville--Bogomolov pairing, which is a generalization of an intersection pairing of K3 surfaces, on $H^{2}$ of an irreducible symplectic variety, and the Kuga--Satake construction also works for irreducible symplectic varieties.
Moreover, recently, Bindt \cite{Bindt2021} describes period maps in terms of moduli stacks.
Therefore, to show Theorem \ref{thmintromain}, we can take a similar approach to the case of K3 surfaces (\cite{Andre1996}, \cite{She2017}, \cite{Takamatsu2020a}),  i.e.\,we first show the Shafarevich conjecture for polarized irreducible symplectic varieties, and then we use the uniform Kuga--Satake construction, which is introduced in \cite{She2017}, \cite{Orr2018}, to show the unpolarized case.

However, to proceed with the proof for general irreducible symplectic varieties, there exist two extra difficulties.
First, for general irreducible symplectic varieties, moduli spaces of polarized varieties may not be fine moduli, since there may exist non-trivial automorphisms acting $H^2$ trivially.
To overcome such a difficulty, we need to show the finiteness of twists admitting smooth reduction. 
To do so, we will show the unramifiedness of Galois cocycles associated with such twists by using Matsusaka--Mumford theorem.

Secondary, to conclude Theorem \ref{thmintromain} from the argument using the uniform Kuga--Satake construction, we need to bound the Beauville--Bogomolov squares of polarizations on irreducible symplectic varieties of bounded discriminants of geometric Picard lattices. 
In the case of K3 surfaces, such a statement can be justified by the action of Weyl groups on Picard lattices, but for general irreducible symplectic varieties, we need additional arguments since there may exist birationally equivalent non-isomorphic irreducible symplectic varieties. 
Such a problem is related to the cone conjecture for irreducible symplectic varieties, and actually already discussed in \cite{Takamatsu2021}.

After I uploaded the first version of this paper on arXiv, this paper was superseded by the joint paper with Lie Fu, Zhiyuan Li, and  Haitao Zou, who proved similar results independently.

The outline of this paper is as follows.
In Section \ref{sectionmoduli}, we recall the definition and properties of  irreducible symplectic varieties. We also recall the definition of moduli spaces and the uniform Kuga--Satake construction.
In Section \ref{sectionMatsusakaMumford}, we formulate Matsusaka--Mumford theorem for algebraic spaces.
In Section \ref{sectiontwists}, we prove the Shafarevich conjecture for twists of polarized  irreducible symplectic varieties, by using the results in Section \ref{sectionMatsusakaMumford}.
In Section \ref{sectionpfmainthm}, we prove the main theorem Theorem \ref{thmintromain}.
In Section \ref{sectionremcoh}, we give counterexamples to the second cohomological formulation of the Shafarevich conjecture for irreducible symplectic varieties, and prove Theorem \ref{thmintromain2}.

%\begin{ack}
%The author is deeply grateful to his advisor Naoki Imai for deep encouragement and helpful advice. Moreover, the author would like to thank Tetsushi Ito for helpful suggestions and valuable discussions. 
%He taught the author Lemma \ref{Lemma:Specialization}, and the idea of Proposition \ref{Proposition:FinitenessTwists2}.
%The author would like to thank Yoichi Mieda for teaching the author an idea using integral models in the proof of Theorem \ref{thmintromain2}, Theorem \ref{cohshaf}.
%The author also thanks Shou Yoshikawa, Kenta Hashizume, Alexei N. Skorobogatov, Yuki Yamamoto, Tatsuro Kawakami for helpful comments and suggestions.
%The author was supported by JSPS KAKENHI Grant number JP19J22795.
%\end{ack}

\subsection*{Acknowledgments}
The author is deeply grateful to my advisor Naoki Imai for his deep encouragement and helpful advice. The author would like to thank Tetsushi Ito for teaching the author Lemma \ref{Lemma:Specialization} and an idea of Proposition \ref{Proposition:FinitenessTwists2}.
The author would like to thank Yoichi Mieda for teaching the author an idea using integral models in the proof of Theorem \ref{thmintromain2} and Theorem \ref{cohshaf}. Moreover, the author would like to thank Yohsuke Matsuzawa, Shou Yoshikawa, Kenta Hashizume, Alexei N. Skorobogatov, Yuki Yamamoto, Takeshi Saito for helpful suggestions.
The author was supported by JSPS KAKENHI Grant number JP19J22795.

\section{Moduli of irreducible symplectic varieties}
\label{sectionmoduli}

\begin{definitionsub}
%\begin{enumerate}
%\item
Let $k$ be a field of characteristic $0$.
%Let $X$ be a smooth projective variety over $k$.
A smooth projective variety $X$ over $k$
is called an \emph{irreducible symplectic} variety if and only if $X_{\overline{k}}$ is simply connected of even dimension $2n$ 
and there exists $\omega \in H^{0}(X, \Omega_{X/k}^{2})$ uniquely up to constant,
which is non-degenerate.
\end{definitionsub}

\begin{defn-propsub}[see also {\cite[Lemma 4.2.1]{Bindt2021}}]
Let $X$ be an irreducible symplectic variety over a field $k$ of characteristic $0$.
Take a finitely generated subfield $k'$ of $k$ where $X$ admits a model $X'$ over $k'$, and choose an embedding $\iota\colon k' \hookrightarrow \C$.
We put $X_{\iota}:= X'\otimes_{k'} \C$ via $\iota$.
By \cite[Proposition 23.14]{Gross2003}, there exists a positive rational number $\alpha_{\iota}$ and a primitive quadratic form $q_{\iota}$ on $H^{2}(X_{\iota}(\C), \Z)$, such that $q_{\iota}$ is of signature $(3, b_{2}(X_{\iota})-3)$, 
\[
q_{\iota}(x)^{n} = \alpha_{\iota} \int_{X_{\iota}} x^{2n}
\]
for any $x \in H^{2}(X_{\iota}(\C), \Z)$, and $q_{\iota}(x)$ is positive for an ample class $x \in H^{2}(X_{\iota}(\C), \Z)$.
By the Artin comparison, we obtain a quadratic form $q_{\iota,\ell}$ on $H^{2}_{\et}(X_{\overline{k}}, \Z_{\ell})$ for any prime number $\ell$.
Then the following holds.
\begin{enumerate}
    \item
    $\alpha_{\iota}$ and $q_{\iota,\ell}$ are independent of the choice of $k'$ and $\iota$.
    \item
    $q_{\iota,\ell}$ is $G_{k}$-equivariant.
\end{enumerate}
We call the quadratic form $q_{\iota,\ell}$ the Beauville--Bogomolov quadratic form on $X$, and $\alpha_{\iota}$ the Fujiki constant of $X$.
Moreover, we denote the non-degenerate even symmetric bilinear pairing associated with $q_{\iota,\ell}$ by $b_{X}$ (this is the $2$-times the usual associated symmetric bilinear form). 
We call $b_{X}$ the Beauville--Bogomolov pairing.  
\end{defn-propsub}

\begin{proof}
First, we shall prove the assertion $(1)$.
The problem is easily reduced to the case $k'=k$.
Let $\iota, \iota' \colon k' \hookrightarrow \C$ be two embeddings.
By \cite[Lemma 2.1.1]{Yang2019}, there exists a real number $r$ such that $q_{\iota,\ell} = r q_{\iota',\ell}$.
By comparing values on ample classes on $X$, we can show that $r$ is independent on $\ell$ and $r \in \Q_{>0}$.
Since $q_{\iota, \ell}$ and $q_{\iota', \ell}$ are primitive, $r$ is an $\ell$-adic unit for any $\ell$.
%By comparing values on the ample class, we have $r^{n} = 1$. %\alpha_{\iota}/\alpha_{\iota'}=1$ for some positive rational numbers $\alpha_{\iota}$ and $\alpha_{\iota'}$. 
Therefore, we have $r=1$.
By comparing values on ample classes on $X$, we also have $\alpha_{\iota} = \alpha_{\iota'}$.
The assertion $(2)$ follows from the same argument as in \cite[Proposition 2.1.5]{Yang2019}.
\end{proof}

%The following notion is introduced by O'Grady.
%
%\begin{defn}
%Numerical Equivalence??
%\end{defn}
%
Now we can introduce the moduli functor of irreducible symplectic varieties.

\begin{definitionsub}[\cite{Gritsenko2010}]
%We fix the numerical equivalence class $N$ of irreducible symplectic varieties.
\begin{enumerate}
    \item 
    Let $S$ be a $\Q$-scheme.
Let $X$ be a proper smooth algebraic space over $S$, and we put $\lambda \in \Pic_{X/S}(S)$.
The pair $(X, \lambda)$ is called a primitively polarized irreducible symplectic family of Beauville degree $d$, if and only if any geometric fiber of $(X, \lambda)$ is an irreducible symplectic variety $X_{s}$ with a primitive ample line bundle $\lambda_{s}$ of Beauville degree $d$, i.e.\,$b_{X_{s}}(\lambda_{s},\lambda_{s})=d$.
We say that $(X, \lambda)$ has the Fujiki constant $c$ if the Fujiki constant of each geometric fiber of $X$ is $c$.
    \item
    Let $n,d,c$ be positive integers.
We define the moduli stack $M_{2n,d, c}$ over $\Q$ as the groupoid fibration whose fiber over a scheme $S$ over $\Q$ consists of primitively polarized irreducible symplectic family of Beauville degree $d$ and Fujiki constant $c$ over $S$ with isomorphisms.
We also define the moduli stack $M_{2n,d,c}'$ over $\Q$ as the groupoid fibration whose fiber over a scheme $S$ over $\Q$ is a subgroupoid of $M_{2n,d,c}(S)$ which consists of objects whose geometric fiber have the second Betti number greater than $3$.
\end{enumerate}
\end{definitionsub}

\begin{propsub}
$M_{2n,d,c}$ is representable by a Deligne-Mumford stack of finite type over $\Q$.
\end{propsub}
\begin{proof}
This essentially follows from \cite[Theorem 1.5]{Gritsenko2010}.
%Indeed, by the argument in \cite[Theorem 1.5]{GHS}, we have a natural smooth surjection $H_{sm} \rightarrow M_{2n,d,N}$, where $H$ denotes some finite union of components of Hilbert scheme, and $H_{sm}$ denotes the smooth part.
%Since $H^{0}(T_{X})= 0$ for irreducible symplectic variety $X$ over $k$, $M_{2n,d,N}$ admits unramified diagonal. Now it finishes the proof. 
Indeed, since $\lambda^{2n}$ is bounded by our assumption, by \cite[Theorem]{Kollar1983}, there are only finitely many possible Hilbert polynomials of objects in $M_{2n,d,c}$ (see also \cite[Theorem 26.15]{Gross2003}).
By the same argument as in \cite[Theorem 1.5]{Gritsenko2010}, by using the smooth part of Hilbert schemes, we obtain a smooth surjection from a quasi-projective scheme of finite type over $\Q$ to $M_{2n,d,c}$.
Since $H^{0}(T_{X})= 0$ for an irreducible symplectic variety $X$ over $k$, $M_{2n,d,c}$ admits an unramified diagonal. Now it finishes the proof. 
\end{proof}

%\begin{defn}
%Let $M_{2n,d}$ be the coarse moduli space of $\mathcal{M}_{2n,d}$.
%
%\end{defn}
We also need the moduli space of oriented polarized irreducible symplectic varieties.
\begin{definitionsub}[{\cite[Definition 4.3.3]{Bindt2021}}]
\begin{enumerate}
\item
Let $S$ be a $\Q$-scheme and $f\colon X \rightarrow S$ a smooth proper morphism of algebraic spaces whose fibers are irreducible symplectic varieties. An orientation on $X/S$ is an isomorphism of sheaves of finite abelian groups
\[
\omega\colon \Z/4\Z \rightarrow \det R^{2}_{\et}f_{\ast}\mu_{4} 
\]
on $S_{\et}$.
\item
We define $\widetilde{M}_{2n,d,c}$ as the groupoid fibration whose fiber over a scheme $S$ consists of $(X \rightarrow S,\lambda, \omega)$, where $(X \rightarrow S, \lambda) \in M_{2n,d,c}'(S)$, and $\omega$ is an orientation on $X \rightarrow S$.
\end{enumerate}
\end{definitionsub}

We recall the definition of certain Shimura varieties.

\begin{defn-propsub}
%Let $M_{2n,d,c,i}$ be a connected component of the moduli stack $\widetilde{M}_{2n,d,c}$.
We denote all the connected components of the moduli stack $\widetilde{M}_{2n,d,c}$ by $\{M_{i}\}_{i\in I_{2n,d,c}}$.
We have the universal object $(f\colon\mathcal{X}\rightarrow M_{i}, \lambda, \omega)$.
Let $x_{0}$ be a $\C$-point of $M_{i}$.
We put the fiber over $x_{0}$ of $(\mathcal{X}\rightarrow M_{i}, \lambda, \omega)$ as $(X_{0}, \lambda_{0}, \omega_{0})$, and we put 
$\Lambda_{0}:= H^{2}(X_{0}, \Z(1))$. %which is a $\Z$-lattice (i.e.\,) by the Beauville--Bogomolov form.
We note that $\Lambda_{0}$ has a $\Z$-lattice structure (i.e.\,$\Z$-valued non-degenerate symmetric bilinear pairing structure) 
by the Beauville--Bogomolov quadratic pairing.
In the following, when we consider isomorphism classes of lattices, we always consider lattice isometry classes.
Then the isomorphism class of $(\Lambda_{0} \otimes \widehat{\Z}, c_{1}(\lambda_{0}), \omega_{0})$ is independent of the choice of $x_{0}$.
%We denote this pair by $(L_{i}, v_{i}, \omega_{i})$.
We denote $\Lambda_{0}$ by $\Lambda_{i}$ (note that this may depend on the choice of $x_{0}$), $c_{1}(\lambda_{0})$ by $v_{i}$, and $\omega_{0}$ by $\omega_{i}$.
We also denote the orthogonal complement $v_{i}^{\perp} \subset \Lambda_{i}$ by $L_{i}$.
\end{defn-propsub}

\begin{proof}
See \cite[Lemma 4.5.1]{Bindt2021}.
\end{proof}

\begin{definitionsub}[{\cite[Theorem 4.5.2]{Bindt2021}}]
\begin{enumerate}
    \item
We put 
\[
D_{i} :=\{ 
g \in \mathrm{SO}_{\Lambda_{i}}(\widehat{\Z}) \mid g(v_{i}) =v_{i}
\},
\]
which can be seen as a subgroup of $\SO_{L_{i}}(\widehat{\Z})$.
We call $D_{i}$ the discriminant kernel.
\item
For any open normal subgroup $K\subset D_{i}$,
we recall the definition of a $K$-level structure on $(f\colon X\rightarrow S, \lambda, \omega)\in M_{i} (S)$.
For simplicity, we assume that $S$ is a locally noetherian connected $M_{i}$-scheme (see the proof of \cite[Theorem 4.5.2]{Bindt2021} or \cite[Definition 3.3.3]{Yang2019}).
We fix a geometric point $s \rightarrow S$.
In this case, a $K$-level structure on $(f\colon X\rightarrow S, \lambda, \omega)$ is a $\pi_{1}^{\mathrm{\et}}(S,s)$-invariant $K$-orbit of an isometry
\[
\Lambda_{i} \simeq H^2_{\et}(X_{s},\widehat{\Z}(1))
\]
which sends $v_{i}$ to $c_{1}(\lambda_{s})$ and $\omega_{i}$ to $\omega_{s}$.
We also define the stack $M_{i,K}$ to be the stack of $M_{i}$ with a level $K$-structure.
%which is defined as a section of the \'{e}tale sheaf 
%\[
%\Isom ((\Lambda_{i}, v_{i},\omega_{i}), (R^{2}_{\et}f_{\ast} \widehat{\Z}(1), c_{1}(\lambda), \omega)
%)/K.
%\]

\end{enumerate}
\end{definitionsub}

\begin{definitionsub}
%Shimura stack...
Let $L$ be a $\Z$-lattice of signature $(2,r)$ with $r\geq 1$.
\begin{enumerate}
    \item 
    We put
    \[
    \Omega_{\SO_{L}}^{\pm}:=
    \{\text{oriented positive definite planes in } L_{\R}
    \}.
    \]
    Then $\Omega_{\SO_{L}}^{\pm}$ is naturally identified with $X_{\SO_{L}}$ which gives the Shimura datum $(\SO_{L}, X_{\SO_{L}})$ with a reflex field $\Q$. 
    For any compact open subgroup $K \subset \SO_{L} (\A_{f})$,
    we write $\Sh_{K}(\SO_{L})$ for the associated Shimura stack of level $K$ (cf.\, \cite[Section 4.4.1]{Bindt2021}).
    \item
    Take a positive definite subspace of $L_{\Q}$, and let $e_{1},e_{2}$ be its orthogonal basis. We put $J:= e_{1}e_{2} \in C^{+}(L_{\R})$, and
    \[
    \psi\colon \mathbb{S} \rightarrow \GSpin_{L,\R}; \alpha + \beta i \mapsto \alpha + \beta J.
    \]
    Let $X_{\GSpin_{L}}$ be a $\GSpin_{L} (\R)$-conjugacy class containing $\psi$.
    Then $(\GSpin_{L}, X_{\GSpin_{L}})$ is a Shimura datum.
    For any compact open subgroup $K \subset \GSpin_{L} (\A_{f})$,
    we write $\Sh_{K} (\GSpin_{L})$ for the Shimura stack of level $K$.
    \item
    Let $V_{L} := C(L)$ be the Clifford algebra.
    We put 
    \[
    \phi \colon V_{L} \times V_{L} \rightarrow \Z ; (x,y) \mapsto \tr_{V_{L}/\Z} (xy^{\ast}).
    \]
    Here, $\tr_{V_{L}/\Z}$ means the trace of a left multiplication map, and $\ast$ denotes the natural anti-automorphism on the Clifford algebra.
    Let $(\GSp_{V_{L}}, X_{\GSp_{V_{L}}})$ be the Shimura datum associated with $(V_{L}, \psi)$.
     For any compact open subgroup $K \subset \GSp_{V_{L}}(\A_{f})$,
    we write $\Sh_{K} (\GSp_{V_{L}})$ for the Shimura stack of level $K$.
\end{enumerate}
\end{definitionsub}

\begin{propsub}
The triple
$(R^{2}f_{\mathrm{an},\ast}\Z(1), c_{1}(\lambda_{\mathrm{an}}), \omega_{\mathrm{an}})$ gives rise to a morphism of complex Deligne-Mumford stacks
\[
j\colon M_{i,\C} \rightarrow \mathrm{Sh}_{D_{i}}(\mathrm{SO}_{L_{i}})_{\C}.
\]
This morphism descends to a morphism 
\[
j \colon M_{i} \rightarrow \mathrm{Sh}_{D_{i}}(\mathrm{SO}_{L_{i}})
\]
defined over $\Q$, and $j$ is \'{e}tale.
Moreover, for any compact open subgroup $K \subset D_{i}$,
there also exist a natural \'{e}tale morphism
\[
j_{K} \colon M_{i,K} \rightarrow \mathrm{Sh}_{K}(\mathrm{SO}_{L_{i}}),
\]
over $\Q$, which is compatible with $j$.
\label{propperiod}
\end{propsub}

\begin{proof}
See \cite[Theorem 4.5.2]{Bindt2021} and its proof.
\end{proof}

Next, to state the main theorem, we introduce the notion of a $\widehat{\Z}$-numerical equivalence class of irreducible symplectic varieties.

\begin{definitionsub}
We fix a positive integer $n$.
%\begin{enumerate}
 %   \item 
Let $k$, $k'$ be fields of characteristic $0$, and $X$ (resp.\,$X'$) a $2n$-dimensional irreducible symplectic variety over $k$ (resp.\,$k'$).
We say $X$ and $X'$ are \emph{$\widehat{\Z}$-numerically equivalent} if there exists an isomorphism of abelian groups
\[
g\colon
H^{2}_{\et}(X_{\overline{k}},\widehat{\Z}) \simeq H^{2}_{\et} (X'_{\overline{k'}}, \widehat{\Z})
\]
such that 
%\[
%\int_{X} x^{2n} = \int_{X'} g(x)^{2n}
%\]
%holds for any $x \in H^{2}_{\et} (X_{\overline{k}}, \widehat{\Z})$ and $g$ preserves the Beauville--Bogomolov pairing.
$g$ preserves the Beauville--Bogomolov pairing and $X$ and $X'$ have the same Fujiki constant.
%\item
%Let $N$ be a $\widehat{\Z}$-numerically equivalent class of $2n$-dimensional irreducible symplectic varieties.
%We say $N$ is \emph{even} if $H^{2}_{\et} (X_{\overline{k}}, \widehat{\Z})$ equipped with Beauville--Bogomolov pairing is even
%for an irreducible symplectic variety $X$ over $k$ with $\widehat{\Z}$-numerically equivalent class $N$.
%\end{enumerate}
\end{definitionsub}

\begin{remarksub}
\label{remnumdef}
Let $k$ be a field of characteristic $0$, and $n$ a positive integer.
\begin{enumerate}
    \item 
Let $X, Y$ be irreducible symplectic varieties over $k$.
Take a finitely generated subfield $k' \subset k$ where $X$ and $Y$ admit models $X'$ and $Y'$ over $k'$.
Take embeddings $\iota_{i} \colon k' \rightarrow \C$ for $i=1,2$.
We put $X_{\C}$ as the base change of $X'$ via $\iota_{1}$, and $Y_{\C}$ as the base change of $Y'$ via $\iota_{2}$.
Suppose that $X_{\C}$ and $Y_{\C}$ are deformation equivalent (see \cite[Theorem 5.0.1]{Takamatsu2021} for the definition).
Then $X$ and $Y$ are $\widehat{\Z}$-numerically equivalent.
   \item
   The converse of (1) holds in the following sense.
Let $N$ be a $\widehat{\Z}$-numerically equivalent class of $2n$-dimensional irreducible symplectic varieties.
Since possible $\Z$-lattice isometry classes of bounded ranks and discriminants are finitely many, there exist $\Z$-lattices $L_{1}, \ldots, L_{m}$ satisfying the following$\colon$
For any irreducible symplectic variety $X$ over $k$ whose $\widehat{\Z}$-numerically equivalent class is $N$, any subfield $k' \subset k$ with an embedding $k'\hookrightarrow \C$, and any irreducible symplectic variety $X'$ over $k'$ with $X'\times_{k'}k\simeq X$, the lattice $H^{2}(X'\times_{k'}\C,\Z)$ is isometric to some $L_{i}$.
Therefore, by using \cite[Corollary 26.17]{Huybrechts1999}, there exist deformation classes of complex irreducible symplectic manifolds $D_{1},\ldots, D_{m'}$ such that for any $k, k'\hookrightarrow \C, X, X'$ as above, the deformation class of $X'\times_{k'}\C$ is one of $D_{1},\ldots, D_{m'}$.
\end{enumerate}
\end{remarksub}

\begin{definitionsub}
We fix a positive integer $n$.
We also fix $N$, a $\widehat{\Z}$-numerically equivalent class of irreducible symplectic varieties.
\begin{enumerate}
\item
Let $M_{2n,d,c,N}$ be a substack of $\widetilde{M}_{2n,d,c}$ whose objects consist of irreducible symplectic families such that each geometric fiber lies on $N$.
Note that this substack is written as a union $\bigcup_{i\in I_{2n,d,c,N}} M_{i}$.
\item
For any $i\in I_{2n,d,c,N}$, the lattice $L_{i}$ is of signature $(3, b_{2}-3)$.
Here, $b_{2}$ is a second Betti number of irreducible symplectic varieties of objects in $M_{i}$, which depends only on $N$.
We put
\[
L_{N} := -( \mathbb{E}_{8}^{\oplus (\lceil \frac{b_{2}-3}{8} \rceil+ \lceil \frac{b_{2}-1}{8}\rceil)} \oplus \mathbb{H}^{\oplus 2}),
\]
which is an even unimodular lattice.
By \cite[Corollary 1.12.3]{Nikulin1979}, for any $d, c$ and $i \in I_{2n,d,c,N}$,
we have a primitive embedding of $\Z$-lattices
$
L_{i} \hookrightarrow L_{N}.
$
\end{enumerate}
\label{deflattice}
\end{definitionsub}

Next, we define the compact open subgroup.
In the following of this paper, we fix an integer $m\geq 3$.

\begin{definitionsub}
\begin{enumerate}
    \item
    For any $i \in I_{2n,d,c}$.
    we put
    \[
    D_{i} (m) := f_{i} (\mathbb{K}_{i,m}^{\mathrm{sp}}),
    \]
    where
    \[
    \mathbb{K}_{i,m}^{\mathrm{sp}} := 
    \{
    g \in \mathrm{GSpin}_{L_{i}} (\widehat{\Z}) \mid
    g=1 \textup{ in } C^{+}(L_{i,\widehat{\Z}/m\widehat{\Z}})
    \},
    \]
    and $f_{i} \colon \GSpin_{L_{i}} \rightarrow \SO_{L_{i}}$ is the natural morphism defined by the conjugation.
    We also put $V_{i}:=V_{L_{i}}$ and
    \[
    \mathbb{K}_{m}:= \{ g \in \GSp_{V_{i}}(\widehat{\Z}) \mid g = 1 \ \text{mod} \ m \}.
    \]
        \item
    Let $N$ be as in Definition \ref{deflattice}.
    We put 
    \[
    D_{N}(m) := f_{N}(\mathbb{K}_{N,m}^{\mathrm{sp}}),
    \]
    where
     \[
    \mathbb{K}_{N,m}^{\mathrm{sp}} := 
    \{
    g \in \mathrm{GSpin}_{L_{N}} (\widehat{\Z}) \mid
    g=1 \textup{ in } C^{+}(L_{N,\widehat{\Z}/m\widehat{\Z}})
    \},
    \]
    and $f_{N}\colon \GSpin_{L_{N}} \rightarrow \SO_{L_{N}}$ is the conjugate morphism.
    We also put $V_{N}:=V_{L_{N}}$ and
    \[
    \mathbb{K}_{m} : = \{ g \in \GSp_{V_{N}} (\widehat{\Z}) \mid g =1 \ \text{mod} \ m \}.
    \]
\end{enumerate}
\end{definitionsub}

%By the argument in \cite[Section 3.3]{Takamatsu2020a}, we have the diagram
%
%\[
%M_{i,D_{i}(m)} \overset{j}{\to} \Sh_{D_{i}(m)}(\SO_{L_{i}})
%  \overset{f_{i}}{\underset{\delta_{i}}{\leftrightarrows}} 
%  \Sh_{\mathbb{K}_{i,m}^{\mathrm{sp}}}(\GSpin_{L_{i}})
%  \overset{h_{i}}{\to} \Sh_{\K_{m}}(\GSp_{V_{i}})
%\]
%for any $i \in I_{2n,d,c}$, provided that $F$ contains certain number field $E_{2n,d,c,m}$.

Fix a $\widehat{\Z}$-numerically equivalent class of irreducible symplectic varieties $N$.
By the argument in \cite[Section 3.3]{Takamatsu2020a}, we have the diagram
\[
\xymatrix{
%M_{i,D_{i}(m)} \overset{j}{\to} \Sh_{D_{i}(m)}(\SO_{L_{i}})
%  \overset{j_{N}}{\to} \Sh_{D_{N}(m)}(\SO_{L_{N}})
%  \overset{f_{N}}{\underset{\delta_{N}}{\leftrightarrows}} 
%  \Sh_{\mathbb{K}_{N,m}^{\mathrm{sp}}}(\GSpin_{L_{N}})
%  \overset{h_{N}}{\to} \Sh_{\K_{m}}(\GSp_{V_{N}})
&&  \Sh_{\mathbb{K}_{N,m}^{\mathrm{sp}}}(\GSpin_{L_{N}}) \ar[r]^-{h_{N}} \ar@<0.5ex>[d]^-{f_{N}}  & \Sh_{\K_{m}}(\GSp_{V_{N}}) \\
 M_{i,D_{i}(m)}  \ar[r]^-{j} & \Sh_{D_{i}(m)}(\SO_{L_{i}}) \ar[r]^-{j_{N}} & \Sh_{D_{N}(m)}(\SO_{L_{N}})\ar@<0.5ex>[u]^-{\delta_{N}}  & \\
}
\]
for any $i\in I_{2n,d,c,N}$, provided that $F$ contains certain number field $E_{2n,N,m}$.
Here, the field $E_{2n,N,m}$ is taken as in \cite[Section 5.5]{Rizov2010} so that we can take a section $\delta_{N}$ of $f_{N}$.
We note that three varieties on the right end (and morphisms between them) depend only on $2n$, $m$, and $N$.

\section{The Matsusaka--Mumford theorem and specialization lemmas}
\label{sectionMatsusakaMumford}
The Matsusaka--Mumford theorem \cite{Matsusaka1964}
states that a birational map between two algebraic varieties
over the field of fractions of a discrete valuation ring
specializes to a birational map between the special fibers
if special fibers are non-ruled.

It is well-known that this theorem remains true for
birational maps between algebraic spaces which are not necessarily schemes.
(Such birational maps appear naturally in the study of
irreducible symplectic varieties. See \cite[Section 4.4]{Liedtke2018}.)
In this section, we give a statement and a brief sketch of the proof
of the Matsusaka--Mumford theorem for algebraic spaces
because we could not find an appropriate reference.

%Recall that,
%for a decent algebraic space $X$,
%%stacks 64.4 最後
%the set $|X|$ is defined as the set of equivalence classes
%of monomorphisms from the spectrums of fields to $X$.
%If $X$ is integral,
%%stacks 70.4.1 
%the notion of a generic point is defined in the same way
%as schemes.
%%stacks 64.11.1
%Let $\kappa(X) \to X$ be a generic point of $X$.
%The field $\kappa(X)$ is called the \textit{function field} of $X$.
We recall that if $X$ is integral and separated, then there exists an open subspace which is a scheme containing a generic point of $X$. Moreover, there exists the largest such open subspace, which is called the schematic locus of $X$.
\begin{definitionsub}
Let $k$ be a field.
Let $X$ be an $n$-dimensional integral separated of finite type algebraic space over $k$.
We say $X$ is ruled if the schematic locus $X'$ of $X$ is ruled, i.e.\,$X'$ is birationally equivalent to $\PP^{1}_{k} \times Y$ for some scheme $Y$ of dimension $n-1$ over $k$.
%Kollar Definition 1.1
\end{definitionsub}

In this section,  let $R$ be a discrete valuation ring,
and $K = \Frac R$ the field of fractions.

\begin{theoremsub}[Matsusaka--Mumford]
\label{ClosedSubset}
Let $\mathcal{X}, \mathcal{Y}$ be integral smooth proper algebraic spaces over $R$.
Suppose that the special fiber $\mathcal{Y}_{s}$ is not ruled.
Let $f \colon \mathcal{X}_{\eta} \overset{\simeq}{\dashrightarrow} \mathcal{Y}_{\eta}$
be a birational map between the generic fibers.
Then there exist closed algebraic subspaces
$V \subset \mathcal{X}$, $W \subset \mathcal{Y}$
with special fibers $V_{s}, W_{s}$ satisfying $V_{s} \neq \mathcal{X}_{s}$, $W_{s} \neq \mathcal{Y}_{s}$
such that
$f$ extends to an isomorphism
\[
  \widetilde{f} \colon 
  \mathcal{X} \setminus V \overset{\simeq}{\to}
  \mathcal{Y} \setminus W.
\]
\end{theoremsub}

\begin{proof}
%Since the notion of generic point and henselian local ring
%exists for algebraic space,
%essentially the same proof as in \cite[Theorem 2]{Matsusaka1964}
%works.
The same proof as in \cite[Theorem 1]{Matsusaka1964} (see also \cite[Theorem 5.4]{Matsumoto2015a}) works, by using the notion of generic point and henselian local rings of decent algebraic spaces. 
\end{proof}
%
%\begin{rem}
%Note that, the Mumford--Matsusaka theorem is usually
%stated for proper smooth schemes over DVR.
%But, in fact, the proof works under a weaker assumption as above.
%(See the proof of \cite[Theorem 5.4]{Matsumoto}.)
%\end{rem}

%\section{Specialization lemmas}

%The following lemma is a slight variant of
%Lemma \ref{Lemma:ClosedSubset}.
%As before,  let $R$ be a discrete valuation ring,
%and $K = \Frac R$ the field of fractions.

%\begin{cor}
%\label{Lemma:ClosedSubset}
%Let $\mathcal{X} \to \Spec R$, $\mathcal{Y} \to \Spec R$
%be proper smooth algebraic spaces
%such that the generic fibers
%$\mathcal{X}_{\eta} = \mathcal{X} \otimes_R K$,
%$\mathcal{Y}_{\eta} = \mathcal{Y} \otimes_R K$
%are geometrically irreducible and have non-negative Kodaira dimension.
%%% definition of Kodaira dimension.... non-trivial.
%%% 
%Let $\mathcal{X}_s$, $\mathcal{Y}_s$
%be the special fibers.
%Let $f \colon \mathcal{X}_{\eta} \overset{\simeq}{\dashrightarrow} \mathcal{Y}_{\eta}$
%be a birational map between the generic fibers.
%Then there exist proper closed subsets
%$V_f \subset \mathcal{X}$, $W_f \subset \mathcal{Y}$
%such that
%$(V_f)_{s} \neq \mathcal{X}_{s}$, $(W_f)_{s} \subset \mathcal{Y}_{s}$,
%and $f$ extends to an isomorphism
%\[ \overline{f} \colon \mathcal{X} \backslash V_f
%\overset{\simeq}{\to} \mathcal{Y} \backslash W_f .\]
%\end{cor}
%
%\begin{proof}
%Matsusaka-Mumford \cite{Matsusaka1964}.
%(Note: $\mathcal{X}_s$ is non-ruled.)
%See also [Matsumoto, Math Z].
%Note : Matsusaka-Mumford for algebraic spaces? Ref?
%\end{proof}
%
%The following Proposition is 

The following lemma is taught to the author by Tetsushi Ito.

\begin{lemmasub}
\label{Lemma:Specialization}
Let $\mathcal{X}$
be an integral proper smooth algebraic space over $R$.
Suppose that the special fiber $\mathcal{X}_{s}$ is not ruled.
Let $\Bir(\mathcal{X}_{\eta})$ be the group of birational automorphisms
of the generic fiber $\mathcal{X}_{\eta}$.
Let $G \subset \Bir(\mathcal{X}_{\eta})$ be a finite subgroup.
\begin{enumerate}
\item There exists a proper closed subspace $V \subset \mathcal{X}$
with $V_s \neq \mathcal{X}_{s}$ and a homomorphism
\[
\psi \colon G \to \Aut(\mathcal{X} \setminus V)
\]
such that $\psi(f)|_{(\mathcal{X}\setminus V)_{\eta}} = f$.

\item Let $\overline{\psi}$ be the composition of the following maps
\[
G \overset{\psi}{\to} \Aut(\mathcal{X} \setminus V)
 \to \Aut((\mathcal{X} \setminus V)_{s})
\]
given by $\overline{\psi}(f) := \psi(f)|_{(\mathcal{X} \setminus V)_{s}}$.
If the characteristic of the residue field of $R$ is $0$,
the map $\overline{\psi}$ is injective.
If the characteristic of the residue field of $R$ is $p > 0$,
the kernel of $\overline{\psi}$ is a $p$-group.
\end{enumerate}
\end{lemmasub}

\begin{proof}
(1) We put $G = \{ f_1=\mathrm{id},\ldots,f_r \}$.
Let $V_{f_i}, W_{f_{i}} \subset \mathcal{X}_s$ be proper closed subsets
associated with $f_i$ given by Theorem \ref{ClosedSubset} (i.e.\,$f_{i}$ extends to an isomorphism $\widetilde{f}_{i}\colon \mathcal{X} \setminus V_{f_{i}} \simeq \mathcal{X} \setminus W_{f_{i}}$).
Here, we put $V_{f_{1}} = \emptyset$.
We put $V':=\bigcup_{i = 1}^{r} V_{f_i} $,
and 
\[
V:= \bigcup_{i=1}^{r} \widetilde{f}_{i}^{-1}(V').
\] 
First, since $|\widetilde{f}_{i}^{-1}(V') \cup V_{f_{i}}|$ is a closed subset of $|\mathcal{X}|$, the subset $V$ is a closed subspace of $\mathcal{X}$.
Moreover, since we have 
\[
|\widetilde{f}_{i}^{-1}( \widetilde{f}_{j}^{-1}(V'))| \subset
|\widetilde{(f_{j}\circ f_{i})}^{-1}(V') \cup V_{f_{j} \circ f_{i}}|,
\]
we have $\widetilde{f}_{i}^{-1}(V) \subset V$.
Since $V' \subset V$, the morphism $\widetilde{f}_{i}$ restricts to the morphism 
\[
f_{i}' \colon
\mathcal{X} \setminus V \rightarrow \mathcal{X} \setminus V.
\]
Since $(f_{i})' \circ (f_{i}^{-1})' =(f_{i}^{-1})' \circ (f_{i})' = \mathrm{id}$,
$f_{i}'$ is an isomorphism, and $V$ satisfies the desired condition.
%64.12.3 みよ

(2)
Take an automorphism $f \in \Ker \overline{\psi}$ in the kernel of $\overline{\psi}$.
Let $\ord(f)$ be the order of $f$.
Assume that $\ord(f)$ is invertible in the residue field of $R$.
We shall show $f$ is trivial.

Since $f$ acts trivially on $(\mathcal{X} \setminus V)_{s}$,
it fixes every closed point $x \in (\mathcal{X} \setminus V)_{s}$.
%It is enough to show that the action of $f$
%on the local ring $\O_{\mathcal{X}, x}$ at $x$ is trivial
%for every closed point $x$ of
%$\mathcal{X}_s \backslash V$.
Let $k(s)$ be the residue field of $R$.
We can take a $\overline{k(s)}$-valued point $x$ of $(\mathcal{X} \setminus V)_{s}$, such that $x$ is contained in the schematic locus of $(\mathcal{X} \setminus V)_{s}$.
It is enough to show that the action of $f$ on the local ring $\oo_{\mathcal{X},x}$ is trivial. 
Let $\pi$ be a uniformizer of $R$.
Since $\pi$ is not invertible in $\oo_{\mathcal{X}, x}$,
we have $\bigcap_{n \geq 1} \pi^n \oo_{\mathcal{X}, x} = (0)$ by Krull's intersection theorem; see \cite[Theorem 8.10]{Matsumura1989}.
Therefore, it is enough to show that
the action of $f$ on 
$\oo_{\mathcal{X}, x}/\pi^n \oo_{\mathcal{X}, x}$
is trivial for every $n \geq 1$.
We shall prove the assertion by induction on $n$.
We put $A_n := \oo_{\mathcal{X}, x}/\pi^{n} \oo_{\mathcal{X}, x}$.
When $n = 1$, it follows because $f$ acts trivially on
the special fiber $(\mathcal{X} \setminus V)_{s}$.
If the assertion holds for some $n \geq 1$, the exact sequence
\[
0 \to \pi^n A_{n+1}
  \to A_{n+1}
  \to A_{n}
  \to 0
\]
shows that, for an element $a \in A_{n+1}$,
we have $f^{\ast}(a) = a + \pi^n b$ for some $b \in A_{n+1}$
(here the action of $f$ on $A_{n+1}$ is denoted by $f^{\ast}$).
Since $\pi^{n+1} = 0$ in $A_{n+1}$,
we have $(f^{\ast})^r(a) = a + r \pi^n b$ for every $r \geq 1$.
In particular, we have
\[ (f^{\ast})^{\ord(f)}(a) = a + \ord(f) \pi^n b = a. \]
Since $\pi^n A_{n+1}$ is a vector space over $k(x)$
and $\ord(f)$ is invertible in the residue field,
we have $\pi^n b = 0$.
Thus, the action of $f$ on $A_{n+1}$ is trivial.
By induction on $n$, the assertion is proved.
\end{proof}

\begin{remarksub}
If $\Aut(\mathcal{X}_{\eta})$ is infinite,
a homomorphism
\[
\Aut(\mathcal{X}_{\eta}) \to \Aut(\mathcal{X} \setminus V)
\]
may not exist for any choice of $V \subset \mathcal{X}_s$.
The problem is there may be infinitely many closed subsets
that should be removed from $\Aut(\mathcal{X})$
if we apply Theorem \ref{ClosedSubset}.
On the other hand, there is a well-defined specialization homomorphism
\[ \Bir(\mathcal{X}_{\eta}) \to \Bir(\mathcal{X}_s), \]
where $\Bir(\mathcal{X}_{\eta})$, $\Bir(\mathcal{X}_s)$
are the groups of birational automorphisms.
See \cite[p.\ 191-192, Exercises 1.17]{Kollar1996} for details.
\end{remarksub}

\begin{remarksub}
There does not exist a specialization map
\[ \Aut(\mathcal{X}_{\eta}) \to \Aut(\mathcal{X}_s) \]
in general.
The problem is that an automorphism of the generic fiber
does not always extend to an automorphism of the integral model.
Such an extension exists for abelian varieties (cf.\,\cite[Theorem 1.4.3]{Bosch1990}).
However, for K3 surfaces, an automorphism does not extend to an automorphism of integral models in general (cf.\,\cite[Example 5.4]{Liedtke2018}). 
%(For concrete examples for $K3$ surfaces, see \cite[Remark? Example?]{LiedtkeMatsumoto})
Note that, since a birational map between K3 surfaces can be extended to an automorphism, there exists a map $\Aut(\mathcal{X}_{\eta}) \to \Aut(\mathcal{X}_s)$ still in this case.
On the other hand, if we suppose that there exists a polarization on $\mathcal{X}$ (this does not hold in general, see \cite[Example 5.2]{Matsumoto2015a}),
there always exists a specialization map
\[
\Aut(\mathcal{X}_{\eta}, \mathcal{L}_{\mathcal{X}_{\eta}}) \to
\Aut(\mathcal{X}_{s}, \mathcal{L}_{\mathcal{X}_{s}}).
\]
(Need to assume the special fiber is non-ruled.)
This map is used in Andr\'e's proof of the Shafarevich
conjecture for (very) polarized irreducible symplectic varieties;
see the proof of \cite[Lemma 9.3.1]{Andre1996}.
\end{remarksub}

\section{Finiteness of twists admitting essentially good reduction}
\label{sectiontwists}
%Let $K$ be a field. Let $X$ be an irreducible symplectic variety over $K$,
%and $\mathscr{L}$ an ample line bundles on $X$.
%The automorphism group as a polarized variety is denoted by
%\[ \mathscr{G} := \Aut(X_{\overline{K}}, \mathscr{L}_{\overline{K}}). \]
%It is a finite group, and $\Gal(\overline{K}/K)$
%acts continuously on it.
%
%Let $(Y, \polM)$ be a polarized irreducible symplectic variety over $K$.
%Let
%$f \colon (X_{\overline{K}}, \mathscr{L}_{\overline{K}})
%\overset{\simeq}{\to}
%(Y_{\overline{K}},\polM_{\overline{K}})$
%be an isomorphism of polarized varieties over $\overline{K}$.
%Recall that we have a $1$-cocycle $\alpha_{f}$ defined as follows:
%for an element $\sigma \in \Gal(\overline{K}/K)$, let
%$f^{\sigma} \colon (X_{\overline{K}}, \mathscr{L}_{\overline{K}})
%\overset{\simeq}{\to}
%(Y_{\overline{K}}, \polM_{\overline{K}})$
%be the conjugate of $f$.
%Then we put
%\[ \alpha_f(\sigma) := (f^{\sigma})^{-1} \circ f \in \mathscr{G}. \]
%The assignment $(Y,L_{Y}) \mapsto [\alpha_{f}]$ gives the bijection
%\[
%\{
%(Y,\polM) | (X_{\overline{K}}, \polL_{\overline{K}} ) \simeq (Y_{\overline{K}}, \polM_{\overline{K}})
%\}/ K\text{-isom}   \simeq H^{1}(\Gal(\overline{K}/K),\Aut(X_{\overline{K}},\polL_{\overline{K}})).
%\]

%For $\sigma, \tau \in \Gal(\overline{K}/K)$.
%we have
%\[
 % \alpha_f(\sigma \tau) = (f^{\sigma \tau})^{-1} \circ f
 % = ((f^{\sigma})^{-1} \circ f)^{\tau} \circ ((f^{\tau})^{-1} \circ f)
 % = \alpha_f(\sigma)^{\tau} \circ \alpha_f(\tau).
%\]

In this section, we prove the Shafarevich conjecture for polarized irreducible symplectic varieties partially, i.e.\,we will prove the finiteness of twists admitting smooth reduction, by using results in Section \ref{sectionMatsusakaMumford}.
First, we introduce the notion of essentially good reduction of irreducible symplectic varieties, which generalizes the notion of good reduction.

\begin{definitionsub}
Let $R$ be a normal domain with fraction field $F$ of characteristic $0$.
Let $X$ be an irreducible symplectic variety over $F$.
Let $\p$ be a height $1$ prime ideal of $\Spec R$.
We say $X$ is essentially good at $\p$ if the following holds.
There exists a finite \'{e}tale extension $S$ of the completion $\widehat{R}_{\p}$ such that there exists a smooth proper algebraic space $\mathcal{Y}$ over $S$ whose generic fiber $\mathcal{Y}_{\Frac(S)}$ is an irreducible symplectic variety which is birational to $X_{\Frac(S)}$.
\label{essentially good}
\end{definitionsub}

\begin{remarksub}
%In \cite{LiedtkeMatsumoto}, 
There is an example of K3 surfaces over $\Q_{p}$ which does not admit good reduction, but admits good reduction after a finite unramified extension of $\Q_{p}$ (see \cite{Liedtke2018}).
\end{remarksub}

\begin{lemmasub}
\label{ruled}
Let $R$ be a discrete valuation ring with $K= \Frac R$.
Let $k(s)$ be the residue field of $R$.
Let $\mathcal{X} \rightarrow \Spec R$ be a smooth proper algebraic space such that the generic fiber $\mathcal{X}_{\eta}$ is an irreducible symplectic variety over $K$.
Then the special fiber $\mathcal{X}_{s}$ is a non-ruled algebraic space over $k(s)$.
\end{lemmasub}

\begin{proof} 
The canonical sheaf
\[
\omega_{\mathcal{X}/R} := \bigwedge \Omega_{\mathcal{X}/R}^{1}
\]
is trivial, since the special fiber $\mathcal{X}_{s}$ is a principal Cartier divisor of $\mathcal{X}$.
% stacks 69.7 そもそも Picard とかって schematic locus で定義しておけばいい気がする Omega もそう． schematic locus の push で定義されることが， flat base change でわかる
%flat base change は https://stacks.math.columbia.edu/tag/073K やぞ
Therefore, the special fiber $\mathcal{X}_{s}$ is a smooth proper algebraic space over $k(s)$ with a trivial canonical sheaf.
Suppose that $\mathcal{X}_{s}$ is ruled.
Take a normal projective variety $V$ over $k(s)$ such that $\PP^{1}_{k(s)} \times V$ is birationally equivalent to $\mathcal{X}_{s}$.
By using Chow's lemma and taking elimination of indeterminacy,
there exists a projective normal scheme $W$ over $k(s)$ with projective birational morphisms $f\colon W \rightarrow \mathcal{X}_{s}$ and $g\colon W \rightarrow \PP^{1}_{k(s)} \times V$.
Take a canonical (Weil)-divisor $K_{W}$ on $W$.
We define a Weil divisor $K_{\mathcal{X}_{s}}$ as the direct image $f_{\ast} (K_{W})$.
We also put $K_{\PP^{1}_{k(s)} \times V} :=g _{\ast} (K_{W}).$
By the \'{e}tale descent, we have $\oo(K_{\mathcal{X}_{s}}) \simeq \omega_{\mathcal{X}_{s}/k(s)}$.
Moreover, we have a natural isomorphism $f_{\ast}(\oo_{W}(K_{W})) \simeq \oo_{\mathcal{X}_{s}}(K_{\mathcal{X}_{s}})$, since $\mathcal{X}_{s}$ has terminal singularity \'{e}tale locally.
On the other hand, we have a natural inclusion
$g_{\ast} (\oo_{W}(K_{W})) \hookrightarrow \oo_{\PP^{1}_{k(s)} \times V} (K_{\PP_{1} \times V}).$
Thus $\Gamma(\PP^{1}_{k(s) \times V},\oo_{\PP^{1}_{k(s)} \times V}(K_{\PP^{1}_{k(s)} \times V})) \neq 0$, and this is contradiction.
\end{proof}

\begin{lemmasub}
\label{Lemma:Unramifiedness}
Let $R$ be a henselian discrete valuation ring with $K = \Frac R$.
Let $(X, \mathscr{L})$, $(Y,\polM)$ be polarized irreducible symplectic varieties over $K$.

Let
$f \colon (X_{\overline{K}}, \mathscr{L}_{\overline{K}})
\overset{\simeq}{\to}
(Y_{\overline{K}}, \polM_{\overline{K}})$
be an isomorphism of polarized varieties over $\overline{K}$
with associated 1-cocycle
\[ \alpha_{f} \colon \Gal(\overline{K}/K) \to \mathscr{G} := \Aut(X_{\overline{K}}, \mathscr{L}_{\overline{K}}) \]
defined by $\alpha_{f}(\sigma) = (f^{\sigma})^{-1} \circ f$.

Assume that
\begin{itemize}
\item the order of $\mathscr{G}$ is invertible in the residue field of $R$,
\item there exist proper smooth algebraic spaces
$\mathcal{X}', \mathcal{Y}'$ over $R$
whose generic fibers are irreducible symplectic varieties
such that there are birational maps
$X \dasharrow \mathcal{X}'_{\eta}$,
$Y \dasharrow \mathcal{Y}'_{\eta}$ over $K$.
\end{itemize}
Then the restriction of $\alpha_{f}$ to the inertia subgroup
$I_K \subset \Gal(\overline{K}/K)$
is trivial.
\end{lemmasub}

\begin{proof}
Take a finite Galois extension $L/K$ such that
$\mathscr{G} = \Aut(X_{L}, \mathscr{L}_{L})$ and
$f$ comes from an isomorphism
$f_L \colon (X_L, \mathscr{L}_{L})
\overset{\simeq}{\to}
(Y_L, \polM_{L})$ over $L$.
The 1-cocycle $\alpha_f$ comes from a $1$-cocycle
$\alpha_{f,L} \colon \Gal(L/K) \to \mathscr{G}$
corresponding to $f_L$.
Let $I_{L/K} \subset \Gal(L/K)$ be the inertia group.
For an element $\sigma \in I_{L/K}$, we have
$\alpha_{f,L}(\sigma) := (f_L^{\sigma})^{-1} \circ f_L$.
We shall show that $\alpha_{f,L}(\sigma) = \id$.

Let $\mathcal{X}'_{\oo_L,s}$, $\mathcal{Y}'_{\oo_L,s}$
be the special fibers of
$\mathcal{X}'_{\oo_L} := \mathcal{X}' \otimes_{\oo_K} \oo_L$,
$\mathcal{Y}'_{\oo_L} := \mathcal{Y}' \otimes_{\oo_K} \oo_L$,
respectively.
We note that the isomorphism $f_L \colon X_L \overset{\simeq}{\to} Y_L$
can be considered as a birational map
$f_L \colon \mathcal{X}'_{L} \overset{\simeq}{\dashrightarrow} \mathcal{Y}'_{L}$.
Here, we put $\mathcal{X}'_{L} := \mathcal{X}' \otimes_{\oo_K} L$ and
$\mathcal{Y}'_{L} := \mathcal{Y}' \otimes_{\oo_K} L$.

By Theorem \ref{ClosedSubset} and Lemma \ref{ruled},
the birational map $f_L$
extends to an isomorphism
\[
  \widetilde{f}_L \colon
  (\mathcal{X}'_{\oo_L} \setminus V) \overset{\simeq}{\to}
  (\mathcal{Y}'_{\oo_L} \setminus W)
\]
over $\oo_L$
for some proper closed subsets
$V \subset \mathcal{X}'_{\oo_L}$, $W \subset \mathcal{Y}'_{\oo_L}$
satisfying
$V_s \neq \mathcal{X}'_{\oo_L,s}$, $W_s \neq \mathcal{Y}'_{\oo_L,s}$.
Enlarging $V$ and $W$ if necessary, by Lemma \ref{Lemma:Specialization},
we may assume the following hold$\colon$
The action of
$\mathscr{G} = \Aut(X_{L}, \mathscr{L}_{L})$ on $X_L$
extends to the homomorphism
\[
  \psi \colon \mathscr{G} \to \Aut(\mathcal{X}'_{\oo_L} \setminus V).
\]
%\end{itemize}

Since $\sigma \in I_{L/K}$ acts trivially on the residue field of $\oo_L$,
the isomorphism
\[
  \widetilde{f}_L \colon
  (\mathcal{X}'_{\oo_L} \setminus V) \overset{\simeq}{\to}
  (\mathcal{Y}'_{\oo_L} \setminus W)
  \quad \text{and} \quad
  \widetilde{f}^{\sigma}_L \colon
  (\mathcal{X}'_{\oo_L} \setminus \sigma(V)) \overset{\simeq}{\to}
  (\mathcal{Y}'_{\oo_L} \setminus \sigma(W))
\]
induce the same morphism on the special fiber.
Thus, the element
\[
 \alpha_{f,L}(\sigma) = (f_L^{\sigma})^{-1} \circ f_L \in \mathcal{G} \subset \Bir (\mathcal{X}_{L})
 \]
sits in the kernel of the specialization map
\[
  \overline{\psi} \colon \mathscr{G} \to \Aut((\mathcal{X}'_{\oo_L} \setminus V)_s)
\]
Since the order of $\mathscr{G}$ is invertible
in the residue field of $R$,
the specialization map
$\mathscr{G} \to \Aut((\mathcal{X}'_{\oo_{L}} \setminus V)_{s})$
is injective by Lemma \ref{Lemma:Specialization} (2).
Therefore,  we have $\alpha_{f,L}(\sigma) = \id$.
\end{proof}

The idea of the proof of the following Proposition was taught to the author by Tetsushi Ito.

\begin{propsub}
\label{Proposition:FinitenessTwists2}
%Let $K$ be a number field, and $S$ a finite set of places containing
%all the infinite places.
Let $F$ be a finitely generated field over $\Q$, and $R$ a finite type algebra over $\Z$ which is a normal domain with fraction field $F$.
Let $X$ be an irreducible symplectic variety over $F$ which admits essentially good reduction at any height $1$ prime $\p \in \Spec R$.
%Let $Y$ be a irreducible symplectic variety over $F$ which is birational to $X$ over $K$.
Let $\polL$ be an ample line bundle on $X$.
Then there exist only finitely many isomorphism classes of
polarized varieties $(Y,\polM)$ over $F$
satisfying the following conditions$\colon$
\begin{itemize}
\item %$Z$ is birational to a irreducible symplectic variety over $K$
%which admits smooth reduction outside $S$.
$Y$ admits essentially good reduction at any height $1$ prime $\p \in \Spec R$.
\item There exists an isomorphism of polarized varieties
$(X_{\overline{F}},\polL_{\overline{F}}) \simeq (Y_{\overline{F}},\polM_{\overline{F}})$
over $\overline{F}$.
\end{itemize}
%(Note: we do \textbf{not} assume $Y, Z$ admit smooth reduction outside $S$.)
\end{propsub}

\begin{proof}
We put $\mathscr{G} := \Aut(X_{\overline{F}}, \mathscr{L}_{\overline{F}})$.
%Shrinking $\Spec R$ if necessary,
%we may assume that the action of $\Gal(\overline{F}/F)$ on $\mathscr{G}$
%is unramified at any height $1$ prime $\p \in \Spec R$.
Shrinking $\Spec R$ if necessary, we may assume that the order of $\mathscr{G}$ is invertible
in $R$.
Note that $\mathscr{G}$ can be considered as a finite subgroup of
$\Bir(X_{\overline{F}})$.

We take a finite Galois extension $L/F$ such that $\Aut(X_{L},\polL_{L}) = \Aut (X_{\overline{F}}, \polL_{\overline{F}})$.
Let $f \colon (X_{\overline{F}}, \polL_{\overline{F}})
\simeq
(Y_{\overline{F}}, \polM_{\overline{F}})$
be any isomorphism of polarized varieties over $\overline{F}$, where $(Y, \polM)$ is any polarized variety satisfying the properties in the statement.
The associated $1$-cocycle
$\alpha_{f} \colon \Gal(\overline{F}/F) \to \mathscr{G}$
is defined by
$\alpha_{f}(\sigma) := (f^{\sigma})^{-1} \circ f$.
It is enough to show that
there are only finitely many possibilities for the $1$-cocycle $\alpha_f$.

Note that $\alpha_{f}|_{\Gal(\overline{F}/L)} \colon \Gal(\overline{F}/L) \rightarrow \mathscr{G}$ is a group homomorphism.
Therefore, the field extension $M/L$ corresponding to the kernel of $\alpha_{f}|_{\Gal(\overline{F}/L)}$ satisfies
$[M:L] < |\mathscr{G}|$.
Let $R'$ be a normalization of $R$ in $L$, and $\p$ a height $1$ prime ideal of $R'$.
By the assumption, there exists a finite \'{e}tale extension $\widehat{R'}_{\p} \subset S$ such that
there exist smooth proper algebraic spaces over $S$ whose generic fibers are irreducible symplectic varieties which are birationally equivalent to $X_{\Frac S}$ and $Y_{\Frac S}$ respectively.
Therefore, by Lemma \ref{Lemma:Unramifiedness}, the inertia subgroup $I_{\p} \subset \Gal(\overline{F}/L)$ (which is defined by fixing the extension of valuation $\p$ to $\overline{F}$) is contained in the kernel of $\alpha_{f}|_{\Gal(\overline{F}/L)}$.
We may assume that $R'$ is regular by shrinking $\Spec R$ if necessarily.
By Zariski--Nagata's purity, $\alpha_{f}|_{\Gal(\overline{F}/L)}$ factors through $\pi_{1}(\Spec R', \overline{F})$.
By \cite[Proposition 2.3, Theorem 2.9]{Harada2009}, the family of subsets
\[
\mathcal{C}:= \{ H \subset \pi_{1}(\Spec R', \overline{F})\colon \text{open subgroup} \mid [\pi_{1}(\Spec R', \overline{F}) : H] \leq |\mathscr{G}|\}
\]
is a finite set.
%We take a finite Galois extension $M'/L$ corresponding to $\bigcap_{H \subset \mathcal{C}} H$, and 
Let $M'$ be the fraction field of the finite \'{e}tale cover of $R'$ corresponding to $\bigcap_{H \subset \mathcal{C}} H$, and $M''$ the Galois closure of $M'$ over $K$. We note that $M''$ does not depend on $(Y,\polM)$. 
Then $\alpha_{f}|_{\Gal(\overline{F}/L)}$ factors through the finite group $\Gal(M''/L)$, so there are only finitely many possibilities for the $1$-cocycle $\alpha_f$. It finishes the proof.
\end{proof}

\section{Proof of the main theorem}
\label{sectionpfmainthm}
First, we shall show the Shafarevich conjecture for polarized irreducible symplectic varieties.

\begin{theoremsub}
Let $F$ be a finitely generated field over $\Q$,
and $R$ a finite type algebra over $\Z$ which is a normal domain with fraction field $F$.
Let $n,d$ be positive integers, and $c$ a positive number.
We denote the set
\[
\left\{
  (X, \lambda) \left|
  \begin{array}{l}
  X \colon 2n \text{-dimensional irreducible symplectic variety over} \ F \\
 b_{2} (X_{\overline{F}}) \geq 4 \\
  X \ \text{has the Fujiki constant} \ c \\
  X \ \text{is essentially good at any height }1 \text{ prime} \ \p\in \Spec R \\
  \lambda \colon \text{polarization of Beauville degree} \ d
  \end{array}
  \right.
\right\}/\text{$F$-isom}
\]
by $\Shaf(F,R,2n,d,c)$.
Then $\Shaf (F,R,2n,d,c)$ is a finite set.
\label{thmpol}
\end{theoremsub}

We fix an embedding $F \rightarrow \C$. %and an integer $m\geq 3$.
Recall that we fix an integer $m \geq 3$.
We may assume $1/m \in R$.
First, we denote the set
\[
%\Shaf(F,R,d,m,\C) :=
\left\{
  (X, \lambda, \omega, \alpha)_{\C} \left|
  \begin{array}{l}
(X, \lambda, \omega, \alpha) \colon \text{an object of} \ \cup_{i\in I_{2n,d,c}} M_{i,D_{i}(m)} (F) \\
X \ \text{is essentially good at any height }1 \text{ prime} \ \p\in \Spec R
  \end{array}
  \right.
\right\}/\text{$\C$-isom}
\]
by $\Shaf(F,R,2n,d,c,m,\C)$.
We shall show the finiteness of $\Shaf(F,R,2n,d,c,m,\C)$.
Since $I_{2n,d,c}$ is a finite set, it suffices to show the finiteness of 
\[
\Shaf(F,R,2n,d,c,m,\C)_{i} \subset \Shaf(F,R,2n,d,c,m,\C)
\]
 which consists of elements that come from objects in $M_{i,D_{i}(m)}$.
We take an object $x = (X, \lambda, \omega,\alpha) \in M_{i,D_{i}(m)}(F)$.
%Let $A_{x}$ be an abelian variety corresponds to $h_{i} \circ \delta_{i} \circ j_{i} (x)$.
Take a $\widehat{\Z}$-numerically equivalent class $N$ such that $i \in I_{2n,d,c,N}$.
We note that we may assume that $F$ contains the number field $E_{2n,N,m}$, which is defined in Section \ref{sectionmoduli}.
Let $A_{x}$ be an abelian variety corresponding to $h_{N} \circ \delta_{N} \circ j_{N} \circ j (x)$.
\begin{lemmasub}
Let $\p \in \Spec R$ be a height $1$ prime ideal.
Suppose that $X$ is essentially good at $\p$.
Then $A_{x}$ admits good reduction at $\p$, i.e.\, there exists a smooth projective model over the localization $R_{\p}$ whose generic fiber is isomorphic to $A_{x}$.
\label{lemKSgood}
\end{lemmasub}
\begin{proof}
By the argument in \cite[Lemma 9.3.1]{Andre1996}, (see also \cite[Proposition 4.2.4]{Takamatsu2020a}),
it suffices to show that
$
H^{2}_{\et} (X_{\overline{k}}, \Z_{\ell})
$
is unramified at $\p$ for some prime number $\ell$ which is different from the residue characteristic of $R_{\p}$
 (here, we use that $m$ is coprime to the residue characteristic of $R_{\p}$, and $m\geq 3$ so that we use the Raynaud semi-abelian reduction criterion).
We take a finite unramified extension $S$ of $\widehat{R_{\p}}$ and a smooth proper algebraic space $\mathcal{Y}$ over $S$ as in Definition \ref{essentially good}.
By the smooth and proper base change theorem, $H^{2}_{\et}(\mathcal{Y}_{\overline{\Frac (S)}}, \Z_{\ell})$ is unramified at $\p$ as a $\Gal(\overline{\Frac (S)}/\Frac (S))$-representation.
Since any birational map between irreducible symplectic varieties is small, we also have a $\Gal(\overline{\Frac (S)}/\Frac (S))$-equivariant isomorphism
\[
H^{2}_{\et}(X_{\overline{\Frac(S)}}, \Z_{\ell}) \simeq H^{2}_{\et}(\mathcal{Y}_{\overline{\Frac (S)}}, \Z_{\ell}).
\]
Therefore, it finishes the proof.
\end{proof}

Now we can show the finiteness of $\Shaf (F,R, 2n,d,c,m,\C)$.

\begin{lemmasub}
$\Shaf (F,R, 2n,d,c,m,\C)$
is a finite set.
Moreover, the set
\[
\left\{
  (X, \lambda, \omega, \alpha) \left|
  \begin{array}{l}
(X, \lambda, \omega, \alpha) \colon \text{an object of} \ \cup_{i\in I_{2n,d,c}} M_{i,D_{i}(m)} (F) \\
X \ \text{is essentially good at any height }1 \text{ prime} \ \p\in \Spec R
  \end{array}
  \right.
\right\}/\text{$F$-isom},
\]
which we shall denote by $\Shaf(F,R,2n,d,c,m)$, is a finite set.
\label{lemthmpol}
\end{lemmasub}

\begin{proof}
%As we stated before, it suffices to show that $\Shaf (F,R,2n,d,c,m,\C)_{i}$ is a finite set.
We will show that $\Shaf (F,R,2n,d,c,m,\C)_{i}$ is a finite set for any $i\in I_{2n,d,c}$.
Take a $\widehat{\Z}$-numerically equivalent class $N$ such that $i \in I_{2n,d,c,N}$.
We may assume that $1/m \in R$.
Moreover, we may assume that $F$ contains the number field $E_{2n,N,m}$ since automorphism groups of polarized irreducible symplectic varieties are finite groups.
By Lemma \ref{lemKSgood}, 
for any $x  \in\Shaf(F,R,2n,d,c,m)_{i}$ (here, $\Shaf(F,R,2n,d,c,m)_{i}$ is defined in a similar way to  $\Shaf (F,R,2n,d,c,m,\C)_{i}$), $A_{x}$ admits a good reduction at any height $1$ prime $\p$.
Therefore, the image 
\[
h_{N} \circ \delta_{N} \circ j_{N} \circ j ( \Shaf(F,R,2n,d,c,m,\C)_{i})
\]
is finite by \cite[VI, \S1, Theorem 2]{Faltings1992}.
Now the finiteness of $\Shaf(F,R,2n,d,c,m,\C)_{i}$ follows from the quasi-finiteness of $h_{N} \circ\delta_{N} \circ j_{N} \circ j$, which holds by Proposition \ref{propperiod}.
Then the finiteness of $\Shaf (F,R,2n,d,c,m)$ follows from Proposition \ref{Proposition:FinitenessTwists2}.
\end{proof}

To give level structures by a uniform field extension of $F$, we need the following lemma.

\begin{lemmasub}
There exists a finite Galois extension $E$ over $F$ such that for every element $(X,\lambda) \in \Shaf (F,R,2n,d,c)$ there exists a level $D_{i}(m)$-structure on $(X,\lambda)_{E}$.
\label{level}
\end{lemmasub}

\begin{proof}
This follows from the Hermite--Minkowski theorem \cite[Proposition 2.3, Theorem 2.9]{Harada2009}.
\end{proof}

\subsection*{Proof of Theorem \ref{thmpol}}
We take a finite Galois extension $E$ as in Lemma \ref{level}.
%Since automorphism groups of polarized irreducible symplectic varieties are finite groups,
Then for any $(X, \lambda_{X}) \in \Shaf (F,R,2n,d,c)$, there exist $\omega_{X}, \alpha_{X}$ such that
$(X_{E},\lambda_{X,E}, \omega_{X}, \alpha_{X}) \in \Shaf (E,R_{E},2n,d,c,m)$, where $R_{E}$ is the normalization of $R$ in $E$.
Therefore, we have a map
\[
\Shaf (F,R,2n,d,c) \rightarrow \Shaf (E,R,2n,d,c,m)
\]
which is finite-to-one since automorphism groups of polarized irreducible symplectic varieties are finite groups.
Therefore, $\Shaf (F,R_{E},2n,d,c)$ is finite by Lemma \ref{lemthmpol}. \qed

The following is the main theorem of this paper.

\begin{theoremsub}
Let $F$ be a finitely generated field over $\Q$,
and $R$ a finite type algebra over $\Z$ which is a normal domain with fraction field $F$.
Let $n$ be a positive integer.
Let $N$ be a $\widehat{\Z}$-numerically equivalent class of $2n$-dimensional irreducible symplectic varieties.
Suppose that irreducible symplectic varieties of $\widehat{\Z}$-numerically equivalent class $N$ have the second Betti number $\geq 5$.
Then the set
\[
\Shaf(F,R,N) :=
\left\{
  X \left|
  \begin{array}{l}
  X \colon\text{irreducible symplectic variety over} \ F \\
 % X \ \text{has Fujiki constant} \ c \\
 \widehat{\Z} \text{-numerically equivalent class of} \ X \ \text{is} \ N \\
  X \ \text{is essentially good at any height }1 \text{ prime} \ \p\in \Spec R \\
%  \lambda : \text{polarization of Beauville degree} \ d
  \end{array}
  \right.
\right\}/\text{$F$-isom}
\]
is a finite set.
\label{mainthm}
\end{theoremsub}
%  In the following in this section, we take $F, R, n. N$ as in Theorem \ref{mainthm}.
  First, we shall show the set
\[
%\Shaf(F,R,d,m,\C) :=
\left\{
  X \left|
  \begin{array}{l}
    X \colon\text{irreducible symplectic variety over} \ F \\
   \widehat{\Z} \text{-numerically equivalent class of} \ X \ \text{is} \ N \\
  \text{there exist} \ d, \lambda, \omega, \alpha \ \text{such that} \\ 
%(X, \lambda, \omega, \alpha) \ \text{an object of} \ \cup_{i\in I_{2n,d,c}}(F) M_{i,D_{i}(m)} (F) \\
(X, \lambda, \omega, \alpha) \in M_{i,D_{i}(m)} \ \text{for some} \ i\in I_{2n,d,c,N} \\
 \Pic_{X/F} (F) \simeq \Pic_{X/F} (\overline{F}) \\ 
X \ \text{is essentially good at any height }1 \text{ prime} \ \p\in \Spec R
  \end{array}
  \right.
\right\}/\text{$F$-isom},
\]
which we denote by $\Shaf (F,R,N,m)$, is a finite set.

\begin{lemmasub}
%The number of isomorphism classes of quadratic lattices  $X \in \Shaf (F,R,N,m)$
Suppose $F$ contains the number field $E_{2n, N, m}$. Then
\[
\{\Pic_{X/F} (F) \mid  X \in \Shaf (F,R,N,m)
\}/ \text{lattice isometry}
\]
is a finite set.
\label{lemfinlattice}
\end{lemmasub}

\begin{proof}
This follows from the same argument as in \cite{She2017}, \cite{Takamatsu2020a}. 
We recall the sketch of arguments.
By the finiteness of lattice isometry classes of lattices bounded rank and discriminant \cite[Chapter 9, Theorem 1.1]{Cassels1982} and Lemma \ref{latticelem}, it suffices to show that
\[
\{
T(X)_{\widehat{\Z}} \mid X \in \Shaf (F,R,N,m)
\}/\text{lattice isometry}
\]
is a finite set.
Here, we put
\[
T (X)_{\widehat{\Z}} = \Pic_{X/F} (F) ^{\perp} \subset H^{2}_{\et} (X_{\overline{F}}, \widehat{\Z}(1)).
\]
For $X \in \Shaf (F,R,N,m)$, we take $d_{X}, \lambda_{X}, \omega_{X}, \alpha_{X}$ such that
\[
x_{X} :=
(X,\lambda_{X},\omega_{X},\alpha_{X})\in M_{i_{X},D_{i_{X}} (m)}
\]
for some $i_{X} \in I_{2n, d_{X},c_{X}, N_{X}}$
By the Tate conjecture for codimension 1 cycles of irreducible symplectic varieties \cite[Theorem 1.6.1]{Andre1996},
we have 
\[
T(X)_{\widehat{\Z}} = (P^{2}_{\et} ((X_{\overline{F}}, \lambda_{\overline{F}}), \widehat{\Z}(1))^{G_{F}})^{\perp},
\]
where we put $G_{F} := \Gal(\overline{F}/F)$ and
\[
P^{2}_{\et}( (X_{\overline{F}}, \lambda_{\overline{F}}), \widehat{\Z}(1)) := c_{1} (\lambda_{\overline{F}})^{\perp} \subset H^{2}_{\et} (X_{\overline{F}}, \widehat{\Z}(1)).
\]
On the other hand, we have 
\[
(P^{2}_{\et} ((X_{\overline{F}}, \lambda_{\overline{F}}), \widehat{\Z}(1))^{G_{F}})^{\perp} = (L_{N,\widehat{\Z}, (X,\lambda_{X},\omega_{X}, \alpha_{X})}^{G_{F}})^{\perp},
\]
where
 $L_{N,\widehat{\Z}, (X,\lambda_{X},\omega_{X}, \alpha_{X})}$ is a $G_{F}$-lattice which is the pull back of the $\widehat{\Z}$-sheaf on $\Sh_{D_{N}(m)}(\SO_{L_{N}})$ corresponding to the representation $D_{N}(m) \rightarrow \SO (L_{N , \widehat{\Z}})$ (see the proof of \cite[Proposition 3.3.3 (1)]{Takamatsu2020a}).
Since the set 
\[
\{j_{N} \circ j(x_{X}) \mid X \in \Shaf (F, R, N,m) \}
\]
is finite by Lemma \ref{lemKSgood}, \cite[VI, \S1, Theorem 2]{Faltings1992}, and quasi-finiteness of $h_{N} \circ \delta_{N}$, we can show that
\[
\{
(L_{N,\widehat{\Z}, (X,\lambda_{X},\omega_{X}, \alpha_{X})}^{G_{F}})^{\perp} \mid X \in \Shaf (F,R,N,m)
\}/\text{lattice isometry}
\]
is a finite set, and it finishes the proof.
\end{proof}

\begin{lemmasub}
Let $M$ be a $\Z$-lattice (i.e.\,a free $\Z$-module of finite rank equipped with an integral non-degenerate symmetric bilinear form $\langle,\rangle$) with discriminant $d$, $N \hookrightarrow M$ be a primitive embedding of lattices.
Let $N^{\vee}$ be a dual lattice of $N$, and we regard $N$ as a sublattice of $N^{\vee}$ via the natural map
\[
N \hookrightarrow N^{\vee}; n \mapsto (x\mapsto \langle x,n \rangle).
\]
We put $N^{\perp}$ as the orthogonal complement of $N$ in $M$.
Then we have 
\[
[N^{\vee}:N] \leq d [ (N^{\perp})^{\vee}:N^{\perp}].
\]
\label{latticelem}
\end{lemmasub}
\begin{proof}
Since $N \hookrightarrow M$ is primitive, we have
%\[
%M/ (N + N^{\perp}) \hookrightarrow M^{\vee} / (N + N_{M^{\vee}}^{\perp}) \simeq N^{\vee}/ %N,
%\]
\[
M^{\vee}
/(N + N_{M^{\vee}}^{\perp}) \simeq N^{\vee}/ N,
\]
where $M^{\vee}$ is the dual lattice of $M$ and $N_{M^{\vee}}^{\perp}$ is the orthogonal complement of $N$ in $M^{\vee}$.
Therefore, we have
%\[
%[N^{\vee}:N] \leq d [M:(N + N^{\perp})].
%\]
\[
[N^{\vee}:N] \leq [M^{\vee}:M][M:(N+N_{M^{\vee}}^{\perp})\cap M] =d [M: (N+N^{\perp})]
\]
Similarly, we have
\[
M/ (N + N^{\perp}) \hookrightarrow M^{\vee} / (N^{\perp} + (N^{\perp})^{\perp}_{M^{\vee}}) \simeq (N^{\perp})^{\vee} / N^{\perp}.
\]
Therefore, we have 
\[
[M: N + N^{\perp}] \leq [(N^{\perp})^{\vee}: N^{\perp}],
\]
and it finishes the proof.
\end{proof}

\begin{lemmasub}
Let $X$ be an irreducible symplectic variety over $F$ with the second Betti number
$ \geq 5$.
Suppose that the action of $\Gal(\overline{F}/F)$ on $\Pic(X_{\overline{F}})$
is trivial.
Then there exists an integer $C_{F,X}$ such that
there exists  a polarization on $Y$ over $F$ of Beauville degree $\leq C_{F,X}$ for any irreducible symplectic variety $Y$ satisfying the following condition:
Irreducible symplectic varieties $Y$ and $X$ are $\widehat{\Z}$-numerically equivalent, and there exists
a $\Gal(\overline{F}/F)$ equivariant isometry 
\[
\phi_{Y}\colon \Pic(X_{\overline{F}}) \simeq \Pic (Y_{\overline{F}}).
\]
%Then there exists a polarization on $Y$ over $F$ of Beauville degree $\leq C_{F,X}$.
\label{lempol}
\end{lemmasub}
\begin{proof}
This follows from \cite[Lemma 4.3.1]{Takamatsu2021} and Remark \ref{remnumdef}.2.
\end{proof}
Now we can show the finiteness of $\Shaf (F,R,N,m)$.
\begin{lemmasub}
$\Shaf (F,R,N,m)$ is a finite set.
\label{lemmainthm}
\end{lemmasub}
\begin{proof}
We may suppose that $F$ contains the number field $E_{2n, N, m}$ by \cite[Theorem 1.0.1]{Takamatsu2021}
By Lemma \ref{lemfinlattice}, Lemma \ref{latticelem} and Lemma \ref{lempol}, there exists an integer $M$ such that
any $X\in \Shaf (F,R, N,m)$ admits a primitive polarization $\lambda_{X}$ of Beauville degree $\leq M$.
Then $X \mapsto (X,\lambda_{X})$ gives an injection
\[
\Shaf (F,R,N,m) \hookrightarrow \bigcup_{d\leq M} \Shaf (F,R,2n,d,c_{N}),
\]
where $c_{N}$ is the Fujiki constant defined by $N$.
By Theorem \ref{thmpol}, $\Shaf (F,R,N,m)$ is a finite set.
\end{proof}

To prove Theorem \ref{mainthm}, we need an unpolarized variant of Lemma \ref{level}.
 
 \begin{lemmasub}

 There exists a finite Galois extension $E$ over $F$ such that for every element $X \in \Shaf (F,R,N)$, the following hold.
 \begin{enumerate}
 \item
 There exists a polarization $\lambda$ on $X$, an orientation $\omega$ on $X_{E}$, a level $m$-structure $\alpha$ on $X_{E}$, and positive integers $d$ such that
 $(X_{E}, \lambda_{E}, \omega, \alpha) \in M_{i,D_{i}(m)} (E)$ for some $i \in I_{2n,d,c_{N},N}$.
 \item
 $\Pic_{X/F}(E) = \Pic_{X/F} (\overline{E})$.
 \end{enumerate}
 \label{unplevel}
 \end{lemmasub}

\begin{proof}
We can take a field $E$ satisfying $(1)$ by the Hermite--Minkowski theorem \cite[Proposition 2.3, Theorem 2.9]{Harada2009} and Lemma \ref{lemlevel}.
We can also take a field $E$ satisfying $(2)$ by the Hermite--Minkowski theorem.
\end{proof}

\begin{lemmasub}%[{cf.\. \cite[Corollary 3.1.8]{Takamatsu}}]
We fix $N$, a $\widehat{\Z}$-numerically equivalent class of irreducible symplectic varieties.
We take an irreducible symplectic variety $X$ over a field $k$ of characteristic $0$, which lies in $N$.
We have
\[
[D_{i}: D_{i} (m)] \leq C m^{2^{(b_{2}(X_{\overline{k}})-2)}},
\]
where we put
\[
C:= [\SO (H^{2}_{\et} (X_{\overline{k}},\Z_{2})): f (\GSpin (H^{2}_{\et} (X_{\overline{k}}. \Z_{2})))],
\]
and $f\colon \GSpin \rightarrow \SO$ is the natural homomorphism defined by the conjugation.
\label{lemlevel}
\end{lemmasub}

\begin{proof}
This follows from the same argument as in \cite[Corollary 3.1.8]{Takamatsu2020a}.
\end{proof}

\subsection*{Proof of Theorem \ref{mainthm}}
We take an $E$ as in Lemma \ref{unplevel}.
Then $X \mapsto X_{E}$ gives a morphism
\[
\Shaf (F,R,N) \rightarrow \Shaf (E,R_{E},N,m),
\]
whose any fiber is finite by \cite[Theorem 1.0.1]{Takamatsu2021}.
Therefore, $\Shaf (F,R,N)$ is finite by Lemma \ref{lemmainthm}.  \hfill $\Box$

\section{Remarks on cohomological generalization}
\label{sectionremcoh}
In this section, we give a counterexample to the second-cohomological generalization of the Shafarevich conjecture. We also discuss other cohomological formulations, by seeing cohomologies of other degrees.

Let $A$ be an abelian surface over a number field $F$ (with a fixed section $0\in A(F)$).
Let $X:=K_{n}(A)$ be the $2n$-dimensional generalized Kummer variety coming from $A$, i.e.\,the fiber $f^{-1}(0)$, where $f\colon A^{[n+1]}\rightarrow A$ is the summation map from $(n+1)$-points Hilbert scheme of $A$.
By extending $F$ if necessary, we may assume that $\Aut(X_{\overline{F}}) = \Aut (X)$ (note that $\Aut(X_{\overline{F}})$ is finitely generated by \cite[Theorem 1.6]{Cattaneo2019}).

By \cite[Theorem 1.2]{Oguiso2020}, the kernel of a natural map
\[
\Aut (X) \rightarrow \GL (H^{2}_{\et}(X_{\overline{F}},\widehat{\Z}))
\]
is the semidirect product $G:=(\Z/2\Z)^{4} \rtimes (\Z/2\Z)$.
Take a subgroup $\Z/2\Z \hookrightarrow G$.
For any surjection $\pi_{i}\colon \Gal(\overline{F}/F) \twoheadrightarrow \Z/2\Z \,\,(i\in I)$\;(note that $I$ is an infinite set), the composition 
\[
\Gal(\overline{F}/F) \twoheadrightarrow \Z/2\Z \hookrightarrow G \rightarrow \Aut(X_{\overline{F}})
\]
gives a different class in the Galois cohomology group $H^{1}(\Gal(\overline{F}/F),\Aut(X_{\overline{F}}))$ (note that the Galois action on $\Aut(X_{\overline{F}})$ is trivial).
We denote the corresponding twist of $X$ by $X_{i}$.
Then $X_{i}$ give infinitely many irreducible symplectic varieties with a Galois isomorphism
\[
H^{2}_{\et}(X_{\overline{F}},\widehat{\Z}) \simeq H^{2}_{\et}(X_{i,\overline{F}},\widehat{\Z}).
\]
In particular, we have the following.

\begin{propsub}
\label{propgenkum}
    Let $\ell$ be any prime number, and $S$ the set of places consist of places $v$ with $v|\ell$ and ramified places of the $\Gal(\overline{F}/F)$-module $H^2_{\et}(X_{\overline{F}},\Z_{\ell})$.
    \begin{enumerate}
    \item 
    The set
    \[
\left\{
  Y \left|
  \begin{array}{l}
  Y \colon\text{irreducible symplectic variety over} \ F \\
 % X \ \text{has Fujiki constant} \ c \\
 Y_{\overline{F}}\simeq X_{\overline{F}} \\
 H^{2}_{\et}(Y_{\overline{F}},\Z_{\ell})\colon \text{unramified outside }S \\
%  \lambda : \text{polarization of Beauville degree} \ d
  \end{array}
  \right.
\right\}/\text{$F$-isom}
\]
is an infinite set.
\item
Let $T$ be any finite set of finite places of $F$.
Then for almost all $X_{i} \,\,(i \in I)$, 
there exists a finite place $v_{i} \notin S \cup T$ such that
$X_{i}$ does not admit essentially good reduction at $v_{i}$ though $H^{2}_{\et}(X_{\overline{F}},\Z_{\ell})$ is unramified at $v_{i}$. 
Moreover, enlarging $T$ if necessary, we may assume that $X_{i}$ admits potentially good reduction at $v_{i}$.
Therefore, the analogue of \cite[Theorem 1.3 (ii)]{Liedtke2018} (see also \cite[Introduction]{Takamatsu2020c}) does not hold for this case.
\end{enumerate}
\end{propsub}

\begin{remarksub}
By using \cite[Theorem 5.2]{Mongardi2017}, we can show that the same statement for twists of $OG_{6}$-type varieties holds.
\end{remarksub}

%On the other　hand, we can show the slightly different formulation of a cohomological generalization holds.
On the other hand, we can show the following.
In the following proof, we use an idea using the integral models of Shimura varieties, which was taught to the author by Yoichi Mieda. 

\begin{theoremsub}
\label{cohshaf}
Let $\ell$ be a prime number, and $I \subset \Z_{\geq 0}$ a finite subset containing $\{2\}$.
Let $F$ be a finitely generated field over $\Q$, and $R$ a finite type algebra over $\Z$ which is a normal domain with fraction field $F$. 
Suppose that $1/\ell \in R$.
For any irreducible symplectic variety $X$ over $F$,
we say that $X$ satisfies (the condition) $C_{I}$ if 
\[
\Aut (X_{\overline{F}}) \rightarrow \GL( \bigoplus_{i\in I} H^{i}_{\et}(X_{\overline{F}},\Q_{\ell}) )
\]
is injective.
Let $n$ be a positive integer, and $N$ a $\widehat{\Z}$-numerically equivalent class of $2n$-dimensional irreducible symplectic varieties with the second Betti number $\geq 5$.
Then the set
\[
S(F,R,N,I,\ell):=
\left\{
  X \left|
  \begin{array}{l}
  X \colon\text{irreducible symplectic variety over} \ F \text{ satisfying }C_{I}\\
 \widehat{\Z} \text{-numerically equivalent class of} \ X \ \text{is} \ N \\
    H^{i}_{\et}(X_{\overline{F}},\Z_{\ell}) \,\,\,(i\in I)\colon \\
  \text{ unramified at any height 1 prime } \p\in \Spec R \\
  \end{array}
  \right.
\right\}/\text{$F$-isom}
\]
is finite.
\end{theoremsub}

\begin{remarksub}
\label{remcohshaf}
\begin{enumerate}
\item
Let $X_{0}$ be a complex irreducible symplectic manifold.
By \cite[Theorem 2.1]{Hassett2013}, the kernel of 
\[
\Aut (X_{0}) \rightarrow \GL(H^{2}(X_{0},\Q))
\]
is a deformation invariant of the manifold $X_{0}.$
\item
By Theorem \ref{cohshaf}, the argument in the proofs of \cite[Theorem 1.3]{Oguiso2020}, \cite[Theorem 5.1]{Oguiso2020}, and \cite[Theorem 3.1]{Mongardi2017}, the set 
\[
\Shaf(F,R,\ast) :=
\left\{
  X \left|
  \begin{array}{l}
  X \colon\text{irreducible symplectic variety over} \ F \text{ of } \ast \\
 % X \ \text{has Fujiki constant} \ c \\
  H^{i}_{\et}(X_{\overline{F}},\Z_{\ell}) \,\,\,(i\in I)\colon \\
  \text{unramified at any height 1 prime } \p\in \Spec R \\
%  \lambda : \text{polarization of Beauville degree} \ d
  \end{array}
  \right.
\right\}/\text{$F$-isom},
\]
is a finite set in the following cases.
\begin{enumerate}
    \item 
    $\ast =$ $K3^{[n]}$-type, and $I=\{2\}$.
    \item
    $\ast =$ generalized Kummer type of dimension $2n$, and $I= \{0,\ldots 4n\}$.
    \item
    $\ast =$ $OG_{10}$-type, and $I= \{2\}$.
\end{enumerate}
Here, we use the convention in \cite[Corollary 6.0.6]{Takamatsu2021}.
\end{enumerate}
\end{remarksub}

\begin{proof} 
Basically, this follows from the same proof as in \cite[Section 4]{Takamatsu2020a}.
However, since we don't know the $\ell$-independence of unramifiedness of cohomologies of irreducible symplectic varieties, we should take a little different approach. 
%We will include the proof where we need additional arguments.

%Since the unpolarized case can be reduced to the polarized case by the same method as in Section \ref{sectionpfmainthm},
First, 
we prove the polarized case, i.e.\,the finiteness of the following set for any positive number $c\colon$
\[
\left\{
  (X, \lambda) \left|
  \begin{array}{l}
  X \colon 2n \text{-dimensional irreducible symplectic variety over} \ F \\
X \text{ satisfies }C_{I}, b_{2}(X_{\overline{F}}) \geq 4 \\
  X \ \text{has the Fujiki constant} \ c \\
    H^{i}_{\et}(X_{\overline{F}},\Z_{\ell}) \,\,\,(i\in I)\colon \\
  \text{unramified at any height 1 prime } \p\in \Spec R \\
  \lambda \colon \text{polarization of Beauville degree} \ d
  \end{array}
  \right.
\right\}/\text{$F$-isom}.
\]
We denote this set by $S_{1} =S_{1}(F,R,2n,d,c,I,\ell)$.
Note that we may assume $1/\ell \in R$ and $R$ is regular.

%By the argument in the proof of \cite[Theorem 1.3]{Oguiso2020}, \cite[Theorem 5.1]{Oguiso2020}, and \cite[Theorem3.1]{Mongardi2017}, it suffices to show the last statement.
%First, we will prove the case where 
%\[
%\Aut (X_{\overline{F}}) \rightarrow \GL (H^{2}_{\et}(X_{\overline{F}},\Q_{\ell}))
%\]
%is injective for any $X$ with $X_{\C}$ is of type $\ast$.
%stacks 92.13.2
%Let $N$ be a connected component
%First, we will prove the case whare $I:=\{2\}$.
In the following, we fix an integer $m$ which is a sufficiently large power of $\ell$.
To avoid using the unramifiedness of $2$-adic cohomologies,
we will introduce compact open subgroups $D_{i}(m)'\subset \SO_{L_{i}}(\widehat{\Z})$, $D_{N}(m)'\subset \SO_{L_{N}}(\widehat{\Z})$ which is slightly different from $D_{i}(m), D_{N}(m)$ in Section \ref{sectionmoduli}.
We put
\[
D_{i}(m)' :=\{g \in D_{i} \mid g =1 \text{ mod } m 
\},
\]
and
\[
D_{N}(m)' :=\{ g \in \SO_{\Lambda_{N}}(\widehat{\Z}) \mid g=1 \text{ mod } m
\}.
\]
As in Section \ref{sectionmoduli}, we have a diagram

\[
\xymatrix{
%M_{i,D_{i}(m)} \overset{j}{\to} \Sh_{D_{i}(m)}(\SO_{L_{i}})
%  \overset{j_{N}}{\to} \Sh_{D_{N}(m)}(\SO_{L_{N}})
%  \overset{f_{N}}{\underset{\delta_{N}}{\leftrightarrows}} 
%  \Sh_{\mathbb{K}_{N,m}^{\mathrm{sp}}}(\GSpin_{L_{N}})
%  \overset{h_{N}}{\to} \Sh_{\K_{m}}(\GSp_{V_{N}})
&&  \Sh_{\mathbb{K}_{N,m}^{\mathrm{sp}}}(\GSpin_{L_{N}}) \ar[r]^-{h_{N}} \ar@<0.5ex>[d]^-{f_{N}}  &\Sh_{\K_{m}}(\GSp_{V_{N}}) \\
 M_{i,D_{i}(m)'} \ar[r]^-{j} &  \Sh_{D_{i}(m)'}(\SO_{L_{i}}) \ar[r]^-{j_{N}} & \Sh_{D_{N}(m)'}(\SO_{L_{N}})  & \\
}.
\]

Note that there is no section of $f_{N}$ in this case.
We put
\[
S_{2}:=
\left\{
  (X, \lambda, \omega, \alpha) \left|
  \begin{array}{l}
(X, \lambda, \omega, \alpha) \ \text{an object of} \ \cup_{i\in I_{2n,d,c}} M_{i,D_{i}(m)'} (F) \\
 H^{2}_{\et}(X_{\overline{F}},\Z_{\ell}) \colon \\
  \text{unramified at any height 1 prime } \p\in \Spec R 
  \end{array}
  \right.
\right\}/\text{$F$-isom}
\]
and
\[
%\Shaf(F,R,d,m,\C) :=
S_{3}:=
\left\{
  (X, \lambda, \omega, \alpha)_{\C} \left|
  \begin{array}{l}
(X, \lambda, \omega, \alpha) \ \text{an object of} \ \cup_{i\in I_{2n,d,c}} M_{i,D_{i}(m)'} (F) \\
%X \text{ satisfies } C_{\{I\}} \\
H^{2}_{\et}(X_{\overline{F}},\Z_{\ell}) \colon \\
  \text{unramified at any height 1 prime } \p\in \Spec R 
  \end{array}
  \right.
\right\}/\text{$\C$-isom}.
\]
First, we will show that $S_{3}$ is a finite set.
By \cite[Section 2, Section 3]{Kisin2010}, the map $f_{N}$ can be extended to the morphism of integral models over $R$
\[
\widetilde{f}_{N}\colon \mathcal{S}_{\mathbb{K}_{N,m}^{\mathrm{sp}}}(\GSpin_{L_{N}}) \rightarrow \mathcal{S}_{D_{N}(m)'}(\SO_{L_{N}}).
\]
%and the left-hand side is the normalization of a closure of  in \mathcal{S}_{\K_{m}}(\GSp_{V_{N}})
We may assume that $\widetilde{f}_{N}$ is a finite \'{e}tale cover.
For any $x\in S_{2}$, let %an closed point $y_{x} \in f_{N}^{-1}(j_{N} \circ j (x))$.
%Then the residue field of $y_{x}$ is has a degree  $F$
 $y_{x} \colon F' \rightarrow \mathcal{S}_{\mathbb{K}_{N,m}^{\mathrm{sp}}}(\GSpin_{L_{N}})$ be the pull back of $j_{N}\circ j (x)$ via $f_{N}$.
 For any height 1 prime $\p$ of $R$, 
 by the proof of Lemma \ref{lemKSgood} and the construction of $\mathcal{S}_{\mathbb{K}_{N,m}^{\mathrm{sp}}}(\GSpin_{L_{N}})$, 
 the map $y_{x}$ can be extended to $\widetilde{y}_{x} \colon \Spec R' \rightarrow \mathcal{S}_{\mathbb{K}_{N,m}^{\mathrm{sp}}}(\GSpin_{L_{N}})$, where $R'$ is the normalization of $R_{\p}$ in $F'$.
Therefore, there exists a morphism $y_{x}' \colon \Spec R_{\p} \rightarrow \mathcal{S}_{D_{N}(m)'}(\SO_{L_{N}}) $ which extends $y_{x}$.
Since the pull back of $y_{x}'$ via $\widetilde{f}_{N}$ is normal, $R'$ is finite \'{e}tale over $R_{\p}$.
By \cite[Proposition 2.3, Theorem 2.9]{Harada2009}, there exists a finite extension $F_{N}/F$ such that for any object $x \in S_{2}$, the $F$-rational point $j_{N} \circ j(x)$ can be lifted to an $F_{N}$-valued point of $z_{x}$ of $\Sh_{\mathbb{K}_{N,m}^{\mathrm{sp}}}(\GSpin_{L_{N}})$.
By \cite[VI, \S1, Theorem 2]{Faltings1992} and the quasi-finiteness of $j_{N} \circ j$, the set $S_{3}$ is finite.

By the choice of the level structure, the unramifiedness condition (note that $\{2\} \subset I$), and the proof of Lemma \ref{level}, we can show the following$\colon$\\
There exists a finite Galois extension $E$ over $F$ such that for every $(X,\lambda) \in S_{2}$, there exists a level $D_{i}(m)'$-structure on $(X,\lambda)_{E}$.

We use the following definition.
\begin{definitionsub}
Let $k$ be a field of characteristic $0$, and $X$ an irreducible symplectic variety over $k$.
We denote the torsion-free part of $\bigoplus_{i \in I} H^{i}_{\et}(X_{\overline{k}},\Z_{\ell})$ by $H_{I,\ell}(X_{\overline{k}})$.
Let $\mathcal{L}$ be the abstract $\Z_{\ell}$-module which is isomorphic to $H_{I,\ell}(X_{\overline{k}})$.
We define a level $\ell^{a}$ structure on $X$ of degrees $I$ as a $\Gal(\overline{k}/k)$-invariant
$\GL(\mathcal{L},\ell^{n})$-orbit $\beta$ of an isomorphism of $\Z_{\ell}$-modules
\[
\mathcal{L} \simeq H_{I,\ell}(X_{\overline{k}}).
\]
Here, we put
\[
\GL(\mathcal{L},\ell^{n}) :=\{g \in \GL (\mathcal{L}) \mid g \equiv 1 \textup{ mod } \ell^{n} \}.
\]
\end{definitionsub}
By the same argument as in the proof of Lemma \ref{level} and Theorem \ref{thmpol}, the problem is reduced to showing that
\[
S_{4}':=
\left\{
  (X, \lambda, \omega, \alpha,\beta) \left|
  \begin{array}{l}
(X, \lambda, \omega, \alpha) \ \text{an object of} \ \cup_{i\in I_{2n,d,c}} M_{i,D_{i}(m)'} (F) \\
\beta\colon \text{level }m\text{ structure on }X \text{ of degrees } I\\
X \text{ satisfies } C_{I} \\
 H^{2}_{\et}(X_{\overline{F}},\Z_{\ell}) \colon \\
  \text{unramified at any height 1 prime } \p\in \Spec R 
  \end{array}
  \right.
\right\}/\text{$F$-isom}
\]
is a finite set.
Since the automorphism group of any object $(X, \lambda, \omega, \alpha,\beta) \in S_{4}'$ is trivial, the finiteness of $S_{4}'$ also follows from the finiteness of $S_{3}$. It finishes the proof of the polarized case.

Next, we will prove the unpolarized case.
By the choice of $D_{i}(m)'$ and Lemma \ref{unplevel}, we can take a finite extension $E/F$ which satisfies the following$\colon$
\begin{enumerate}
\item
$E$ contains the above field $F_{N}$.
\item
For any element $X \in S(F,R,N,I,\ell)$, there exist a polarization $\lambda$ on $X$, an orientation $\omega$ on $X_{E}$, a level $D_{i}(m)'$ structure $\alpha$ on $X_{E}$, and a positive integer $d$ 
such that
 $(X_{E},\lambda_{E},\omega,\alpha) \in M_{i, D_{i}(m)'}(E)$ for some $i \in I_{2n,d,c_{N},N}$.
 \item
 For any element $X \in S(F,R,N,I,\ell)$, we have $\Pic_{X/F}(E) = \Pic_{X/F} (\overline{F})$.
\end{enumerate}
By the proof of Lemma \ref{lemfinlattice},
the set
\[
\{
\Pic_{X/F}(E) \mid X \in S(F,R,N,I,\ell)
\}/ \text{lattice isometry}
\]
is a finite set.
By Lemma \ref{lempol} and \cite[Theorem 1.0.1]{Takamatsu2021}, the desired finiteness follows from the finiteness of $S_{1}$.
\end{proof}

%\begin{theoremsub}
%Let $K$ be a mixed characteristic discrete valuation field, $k$ be its residue field. Assume $k$ has characteristic away from $2, 3$. Let $X$ be a bielliptic surface over $K$ which admits a section $x\in X(K)$. Let $X \simeq A\times B /G$ be the isomorphism given in Lemma \ref{structure}. Assume $A$ admits good reduction. Then, $X$ admits a N\'{e}ron model.
%\end{theoremsub}
%
%\begin{proof}
%Let $\mathcal{A}$, $\mathcal{B}$, $\mathcal{G}$ be a N\'{e}ron model of $A$, $B$, $G$. Since $A$ admits good reduction, $\mathcal{G}$ is finite \'{e}tale group scheme over $\oo_{K}$. 
%By the N\'{e}ron mapping property, we have the action $\sigma \colon \mathcal{G} \times \mathcal{A} \times \mathcal{B} \rightarrow \mathcal{A} \times \mathcal{B}$. Since this action on $\mathcal{A}$ is given by the translation by closed subgroup scheme $\mathcal{G} \hookrightarrow \mathcal{A}$, $\sigma$ is a free action. Let $\overline{\mathcal{}}$
%\end{proof}
%
%\bibliographystyle{amsalpha_url-2}
%\bibliography{myref7}

\newcommand{\etalchar}[1]{$^{#1}$}
\providecommand{\bysame}{\leavevmode\hbox to3em{\hrulefill}\thinspace}
\providecommand{\MR}{\relax\ifhmode\unskip\space\fi MR }
\providecommand{\MRhref}[2]{%
  \href{http://www.ams.org/mathscinet-getitem?mr=#1}{#2}
}
\providecommand{\href}[2]{#2}

\end{document}